\documentclass[12pt,twoside,reqno]{amsart}
\usepackage[colorlinks=true,citecolor=blue]{hyperref}
\usepackage{mathptmx, amsmath, amssymb, amsfonts, amsthm, enumerate, color}
\usepackage[T1]{fontenc}
\usepackage{graphicx}
\usepackage{epstopdf}

\usepackage{csquotes}
\usepackage[shortlabels]{enumitem}
\usepackage[mathscr]{euscript}
\usepackage{scalerel}

\makeatletter
\newenvironment{sqcases}{%
  \matrix@check\sqcases\env@sqcases
}{%
  \endarray\right.%
}
\def\env@sqcases{%
  \let\@ifnextchar\new@ifnextchar
  \left[
  \def\arraystretch{1.2}%
  \array{@{}l@{\quad}l@{}}%
}
\makeatother

\newtheorem{theorem}{Theorem}[section]

\newtheorem{lemma}{Lemma}[section]

\theoremstyle{definition}

\newtheorem{remark}{Remark}[section]

\numberwithin{equation}{section}

\def\f{\longrightarrow}

\def\N{\mathbb{N}}

\def\e{\varepsilon}
\def\l{\lambda}

\def\x{\bar{x}}

\def\u{\bar{u}}
\def\p{\bar{p}}
\def\xxi{\bar{\xi}}
\def\nnu{\bar{\nu}}
\def\zzeta{\bar{\zeta}}
\def\ttheta{\bar{\theta}}
\def\vvartheta{\bar{\vartheta}}
\def\oomega{\bar{\omega}}
\def\a{\alpha}

\def\<{\langle}
\def\>{\rangle}

\def\e{\varepsilon}
\def\R{\mathbb{R}}

\def\inte{\textnormal{int}\,}
\def\clo{\textnormal{cl}\,}
\def\epi{\textnormal{epi}\,}
\def\bdry{\textnormal{bdry}\,}
\def\dom{\textnormal{dom}\,}
\def\conv{\textnormal{conv}\,}

\def\Gr{\textnormal{Gr}\,}
\def\gk{\gamma_k}
\def\gkn{\gamma_{k_n}}
\def\Vphi{\varPhi}

\def\CO{\mathsf{C}}

\def\-{\textnormal{-}}
\newcommand*{\tran}{^{\mkern-1.5mu\mathsf{T}}}

\def\sp{\hspace{0.015cm}}
\def\bp{\hspace{-0.08cm}}
\def\bbp{\hspace{-0.04cm}}

\def\bbbp{\hspace{-0.02cm}}
\def\TV{_{\textnormal{\tiny{T.\sp V.}}}}
\def\C{\mathrm{C}}
\def\W{\mathrm{W}}
\def\t{\scaleto{\times}{6pt}}

\def\0gk{\scaleto{0,\gk}{3.5pt}}
\def\supp{\textnormal{supp}\,}
\def\z{\bar{z}}

\begin{document}
\setcounter{page}{1}

\vspace*{1.0cm}
\title[Nonsmooth optimality criterion for $\W^{1,2}$-controlled sweeping process]
{Nonsmooth optimality criterion for a $\W^{1,2}$-controlled sweeping process: Nonautonomous perturbation}
\author[C. Nour, V. Zeidan]{ Chadi Nour$^{1}$, Vera Zeidan$^{2,*}$}
\maketitle
\vspace*{-0.6cm}

\begin{center}
{\footnotesize {\it
$^1$Department of Computer Science and Mathematics, Lebanese American University, Byblos Campus,\\P.O. Box 36, Byblos, Lebanon\\
$^2$Department of Mathematics, Michigan State University,  East Lansing, MI 48824-1027, USA

}}\end{center}

\vskip 4mm {\small\noindent {\bf Abstract.} We extend to {\it $\C\t \W^{1,2}$-local minimizers} and {\it nonautonomous} perturbation function, the necessary optimality conditions derived in \cite{ZNpaper}, via continuous-time approximations, for {\it $\W^{1,2}\t \W^{1,2}$-local minimizers} of an optimal control problem governed by a controlled nonconvex sweeping process with {\it autonomous} perturbation function.\\
\noindent{\bf Keywords.} 
Controlled and perturbed nonconvex sweeping process; Local minimizers; Necessary optimality conditions; Continuous-time approximations; Nonsmooth analysis.}

\renewcommand{\thefootnote}{}
\footnotetext{ $^*$Corresponding author.
\par
E-mail addresses: cnour@lau.edu.lb (C. Nour), zeidan@msu.edu (V. Zeidan).
\par
Received; Accepted}

\renewcommand{\thefootnote}{\arabic{footnote}}

\section{Introduction} Moreau's {\it sweeping process} is a mathematical model introduced in  \cite{moreau1,moreau2,moreau3} to describe an elastoplastic mechanical system. Since then, this model  appeared in several fields such as mechanics, electrical circuits, engineering, economics, crowd motion control, traffic equilibria, hysteresis phenomena, etc. (see, e.g., \cite{outrata} and its references). Since the dynamic of this model is an evolution differential inclusions  involving an unbounded and discontinuous multifunction, namely, the {\it normal cone} to a set, the sweeping process model falls outside the scope of standard differential inclusions. Therefore, new techniques are required to address optimal control problems over sweeping processes.

In this paper, we are interested in deriving necessary optimality conditions,  phrased in terms of the weak-Pontryagin-type maximum principle, for optimal control problems governed by $\W^{1,2}$-controlled and {\it nonautonomously} perturbed sweeping processes. The autonomous case  was successfully treated in \cite{cmo0,cmo,cmo2} via {\it discrete approximations}, and in \cite{ZNpaper} via {\it continuous approximations}. In these references, the authors considered {\it $\W^{1,2}\t \W^{1,2}$-local minimizers}, and obtained in addition to necessary optimality  conditions,  the {\it strong} convergence of velocities, that is, the optimal states of the original problem are {\it strongly} approximated, in the $\W^{1,2}$-norm, by the optimal states of the approximating problems. It is worth to mention that,
unlike the case for standard optimal control problems where the optimal control functions have {\it weak} regularity properties, applications of optimal control problems over {\it sweeping processes} have shown to possess optimal control functions having {\it strong} regularity properties, such as, $\W^{1,2}$ (see  e.g., \cite{ cmo}, \cite[Examples 9.1 \& 9.2]{cmo2}).

The continuous approximation used in \cite{ZNpaper} is based on approximating the normal cone by an {\it exponential penalization} term. This innovative technique was introduced in \cite{pinho,pinhoEr}, and further developed in \cite{pinhonum,pinholast,verachadi,verachadiimp,verachadinum}.  The {\it autonomous} assumption on the perturbation function in \cite{ZNpaper}, and  also in \cite{cmo0,cmo,cmo2},  played a crucial role in the derivation of the necessary optimality conditions using  $W^{1,2}$-controls, especially, in the proof of the strong convergence of velocities. 

The goal of this paper is to extend the weak Pontryagin maximum principle derived in \cite{ZNpaper} to the case where the perturbation function is {\it nonautonomous} and the $t$-dependence is merely {\it measurable}. In order to reach this goal, 
it turns out that the strong convergence of velocities has to be weakened in the following sense:  While the convergence of  the approximating problems solutions  to those  of the original problem remains  {\it uniform} for the {\it state component} and  {\it strong} in $\W^{1,2}$ for the {\it control component},   it is rather  {\it weak*} in $L^{\infty}$  for the state {\it velocities}. Consequently, our necessary optimality conditions, which coincide with those obtained in \cite{ZNpaper}, are derived for {\it $\C\t \W^{1,2}$-local minimizers}. For this class of local minimizers, our  formulation  here of the approximating problems  via the exponential penalization technique  is different than that in \cite{ZNpaper}, and so is the proof of existence of optimal solutions for these problems (Lemma \ref{lasthopefully}) as well as the proof of Lemma \ref{mpw12aprox}.

The paper is organized as follows. In the next section, we present our basic notations, we define our optimal control problem $(P)$ governed by a $\W^{1,2}$-controlled and perturbed sweeping process, we list the hypotheses satisfied by the data of $(P)$, and we state the main result of the paper, namely, necessary optimality conditions in the form of weak Pontryagin principle for $\C\t\W^{1,2}$-local minimizers of $(P)$. Section \ref{proofmn} is devoted to the proof of our main result. An example illustrating the utility of our main result is provided in Section \ref{example}.

\section{Main result}

\subsection{Basic notations} The following are the basic notations and definitions used in this paper:
\begin{itemize}
\item For the Euclidean norm and the usual inner product, we use $\|\cdot\|$ and $\<\cdot,\cdot\>$, respectively. 
\item For $y\in\R^n$ and $r>0$, we define the open (resp. closed) ball centered at $y$ with radius $r$ as $B_{r}(y):=y+rB$ (resp. $\bar{B}_{r}(y):=y+r\bar{B}$), where $B$ and $\bar{B}$ denotes the open and the closed unit ball, respectively. 
\item For $C\subset\R^n$, the boundary, the interior, the closure, the convex hull, the complement, and the polar of $C$ are denoted  by $\bdry C$, $\inte C$, $\clo C$, $\conv C$, $C^c$, and $C^\circ$, respectively. 
\item The distance from a point $x\in\R^n$ to a set $C\subset \R^n$ is denoted by $d(x,C)$.
\item For an extended-real-valued function $h\colon\R^n\f\R\cup\{\infty\}$, the effective domain of $h$ is $\dom h$, and the epigraph of $h$ is $\epi h$.
\item For a multifunction $F\colon\R^n\rightrightarrows\R^{m}$, we denote by $\Gr F\subset \R^n\times\R^m$ the graph of $F$. 
\item \sloppy The space $L^p([a,b];\R^n)$ designates the Lebesgue space of $p$-integrable functions $h\colon [a,b]\f\R^n$. We denote by $\|\cdot\|_p$ and $\|\cdot\|_{\infty}$ the norms of $L^p([a,b];\R^n)$ and $L^{\infty}([a,b];\R^n)$ (or $\C([a,b];\R^n)$), respectively. For $C\subset\R^d$ compact, the set of  continuous functions from  $C$ to $\R^n$ is denoted by $\C(C;\R^n)$.  
\item The set of all $m\times n$-matrix functions  on $[a,b]$ is denoted by $\mathscr{M}_{m\times n}([a, b])$. 
\item A function $h\colon[a,b]\f\R^n$ is said to be a $BV$-function if $h$ has a  bounded variation. The set of all such functions is denoted by  $BV([a,b];\R^n)$. We denote by $NBV[a,b]$ the normalized space of $BV$-functions on $[a,b]$ that consists of those $BV$-functions $h$ such that $h(a)=0$ and $h$ is right continuous on $(a,b)$ (see e.g., \cite[p.115]{luenberger}). 
\item The space $\C^*([a,b]; \R)$ denotes the dual of $\C([a,b];\R),$ equipped with the supremum norm. The induced norm on $\C^*([a,b]; \R)$ is denoted by  $\|\cdot\|\TV$. As a consequence of Riesz representation theorem, we can interpret the elements of $\C^*([a,b]; \R)$ as being in $\mathfrak{M}([a,b])$, the set of finite signed Radon measures on $[a,b]$ equipped with the weak* topology. Thereby,  to each element of $\C^*([a,b]; \R)$ it corresponds a unique element in $NBV[a,b]$ related through the Stieltjes integral and both elements  have the same total variation. The set $\C^{\oplus}(a,b)$ designates the subset of  $\C^*([a,b];\R)$  taking nonnegative values on nonnegative-valued functions in $\C([a,b]; \R)$. 
\item By $\W^{k,p}([a,b];\R^n)$, $k\in\N$ and $p\in [0,+\infty]$, we denote the classical Sobolev space. Hence, the set of all absolutely continuous functions from $[a,b]$ to $\R^n$ is $\W^{1,1}([a,b];\R^n)$. Note that in this paper, the Sobolev space $\W^{1,2}([a,b];\R^m)$ will be considered with the norm $\|u(\cdot)\|_{W^{1,2}}:=\|u(\cdot)\|_\infty + \|\dot{u}(\cdot)\|_2.$ Hence, the convergence of a sequence $u_n$ strongly in the norm topology of the space $W^{1,2}([a,b];\R^m)$ is equivalent to the uniform convergence of $u_n$ on $[a,b]$ and the strong  convergence  in $L^2$ of its derivative $\dot{u}_n$.
\item For $C\subset\R^n$ closed and $c\in C$, we denote by $N^P_C(c)$, $N_C^L(c)$, and $N_C(c)$, the {\it proximal}, the {\it Mordukhovich} (or {\it limiting}), and the {\it Clarke normal cones} to $C$ at $c$, respectively.  
\item For $F\colon [a,b]\rightrightarrows\R^m$ a multifunction with closed and nonempty values, $\bar{N}^L_{F(t)}(y)$ stands for the {\it graphical closure} at $(t,y)$ of the multifunction $(t,y)\mapsto{N}^L_{F(t)}(y)$, that is, the graph of $\bar{N}^L_{F(\cdot)}(\cdot)$ is the closure of the graph of $N^L_{F(\cdot)}(\cdot)$. 
\item For $h\colon\R^n\f\R\cup\{\infty\}$ lower semicontinuous and $x\in\dom h$, we denote by  $\partial^P h (x)$, $\partial^L h(x)$, and $\partial h(x)$ the \textit{proximal}, the {\it Mordukhovich} (or \textit{limiting}), and the \textit{Clarke subdifferential}  of $h$ at  $x$, respectively. Note that if $h$ is Lipschitz near $x$, then the {\it Clarke generalized gradient} of $h$ at $x$ is also denoted by $\partial h(x)$.
\item If $h\colon\R^n\f\R\cup\{\infty\}$ is $\CO^{1,1}$ near $x\in\dom h$, then the  {\it Clarke generalized Hessian} of $h$ at $x$ is denoted by $\partial^2 h(x)$. On the  other hand, if  $H\colon\R^n\f\R^n$ is Lipschitz near $x\in\R^n$, then the {\it Clarke generalized Jacobian} of $H$ at $x$ is denoted by $\partial{H}(x)$.
\item For $F\colon [a,b]\rightrightarrows\R^m$ a lower semicontinuous multifunction with closed and nonempty values, we define
$$\hspace*{1.3cm} \partial_x^{\sp >}d(x,F(t)):= \conv \left\{\zeta=\lim_{i\f\infty}\zeta_i: \;\|\zeta_i\|=1,\;\zeta_i\in N_{F(t_i)}^P(x_i)\;\hbox{and}\;(t_i,x_i)\xrightarrow{{\Gr F}} (t,x)\right\},$$
 where $(t_i,x_i)\xrightarrow{{\Gr F}} (t,x)$ signifies that $(t_i,x_i)\f (t,x)$ with $x_i\in F(t_i)$ for all $i$. Note that $\partial_x^{\sp >}d(x,F(t))$ coincides with $\partial_x^{\sp >}g(t,x)$ of \cite[p.121]{clarkeold} for $g(t,x):=d(x,F(t))$, see  \cite[Corollary 2.2]{lowen91}.
\item For $r>0$, a closed and nonempty set $S\subset \R^n$ is said to be $r$-prox-regular if for all $s\in\bdry S$ and for all $\zeta\in N_S^P(s)$ unit, we have $\<\zeta,x-s\>\leq\frac{1}{2r}\|x-s\|^2$ for all $x\in S.$ For more information about prox-regularity, see \cite{prt}.
\end{itemize}

\subsection{Statement of the problem $(P)$ and hypotheses} We consider the following fixed time Mayer-type optimal control problem involving $\W^{1,2}$-controlled and perturbed sweeping systems
$$\begin{array}{l} (P)\colon\; \hbox{Minimize}\;g(x(0),x(1))\vspace{0.1cm}\\ \hspace{0.9cm} \hbox{over}\;(x,u)\in \W^{1,1}([0,1];\R^n)\times \mathscr{W} \;\hbox{such that}\;\\[2pt] \hspace{0.9cm}  \begin{cases} (D)  \begin{sqcases}\dot{x}(t)\in f(t,x(t),u(t))-\partial \varphi(x(t)),\;\;\hbox{a.e.}\;t\in[0,1],\\x(0)\in C_0\subset \dom \varphi, \end{sqcases}\vspace{0.1cm}\\ x(1)\in C_1, \end{cases}
    \end{array}$$ where, $g\colon\R^n\times\R^n\f\R\cup\{\infty\}$, $f\colon[0,1]\times\R^n\times \R^m \f\R^n$, $\varphi\colon\R^n\f\R\cup\{\infty\}$, $\partial$ stands for the Clarke subdifferential, $C:=\dom\varphi$ is the zero-sublevel set of a function $\psi\colon\R^n\f\R$, that is, $C=\{x\in\R^n : \psi(x)\leq0\},$ $C_0\subset C$, $C_1\subset\R^n$, and, for $U\colon[0,1]\rightrightarrows\R^{m}$ a multifunction and $ \mathbb{U}:=\bigcup_{t\in[0,1]} U(t)$,  the set of control functions $\mathscr{W}$ is defined by 
\begin{equation*} \label{setw}\hspace{-0.2cm}
\mathscr{W}:=\W^{1,2}([0,1]; \mathbb{U})=\left\{ u\in \W^{1,2}([0,1]; \R^m): u(t) \in U(t),\;\; \forall\sp t\in [0,1]\right\}. 
\end{equation*}
Note that if $(x,u)$ solves  ($D$), it necessarily follows that $x(t)\in C$, $\forall\sp t\in[0,1]$.

A pair $(x,u)$ is {\it admissible} for $(P)$ if $x\colon[0,1]\f\R^n$ is absolutely continuous, $u\in \mathscr{W}$, and $(x,u)$ satisfies the  {\it controlled\,} and {\it perturbed} {\it sweeping process} $(D)$, called the {\it dynamic} of $(P)$. An admissible pair $(\x,\u)$ for $(P)$ is said to be a {\it $\C\t\W^{1,2}$-local minimizer} if there exists $\delta >0$ such that \begin{equation}\label{optimal}g(\x(0),\x(1))\le g(x(0),x(1)), \end{equation}
for all $(x,u)$ admissible for $(P)$ with $\|x- \x\|_{\infty}\leq\delta$, $\|u- \u\|_{\infty}\leq\delta$ and $\|\dot{u}- \dot{\u}\|_{2}^2\leq\delta$. Note that if the inequality \eqref{optimal} is satisfied by any admissible pairs $(x,u)$, then $(\x,\u)$ is called a {\it global minimizer} (or an {\it optimal solution}\sp) for $(P)$.

Let  $(\x,\u)$ be a $\C\t\W^{1,2}$-local minimizer for $(P)$  with associated $\delta$, such that the following hypotheses hold for $\bar{\mathbb{B}}_{\delta}(\x):=\bigcup\limits_{t\in [0,1]} \bar{B}_{\delta}(\x(t))$:
\begin{enumerate}[label=\textbf{H\arabic*}:]
\item There exist $\tilde{\rho}>0$ and $M_\ell>0$ such that $f(\cdot, x,u)$ is Lebesgue-measurable for  $(x,u)\in [C\cap \bar{\mathbb{B}}_{\delta}(\x)]\times [(\mathbb{U}+\tilde{\rho}\bar{B}) \cap \bar{\mathbb{B}}_{\delta}(\u)]$; and  for a.e. $t\in  [0, 1]$ we have: $(x,u)\mapsto f (t, x, u)$ is $M_\ell$-Lipschitz on $\left[C \cap \bar{B}_{\delta}(\x(t))\right]\times \left[\big(U(t)+\tilde{\rho}\bar{B}\big)\cap  \bar{B}_{\delta}(\u(t))\right]$; and  $\|f(t,x,u)\| \leq M_\ell$  for all $(x,u)\in \left[ C \cap \bar{B}_{\delta}(\x(t))\right] \times \left[U(t)\cap \bar{B}_{\delta}(\u(t))\right].$ 
\item The set $C:=\dom \varphi$ is given by  $C=\{x\in\R^n : \psi(x)\leq 0\},$ where $\psi:\R^n\f\R.$
\begin{enumerate}[label=\textbf{H2.\arabic*}:]
\item There exists $\rho>0$ such that $\psi$ is $\CO^{1,1}$ on $ C+\rho{B}$.
\item There is a constant $\eta>0$ such that $\|\nabla\psi (x)\|>2\eta\;$ for all $x$ satisfying $\psi(x)=0.$
\item The function $\psi$ is coercive, that is, $\lim_{\|x\|\f\infty} \psi(x)=+\infty.$\footnote{This hypothesis is only required to  guarantee that $C$ is compact. Hence, (H2.3) can be replaced by the boundedness of $C$.}
\item The set $C$ has a connected interior.\footnote{This hypothesis is only imposed to obtain the extension function $\Vphi$ of $\varphi$, see \cite[Remark 3.2 \& Lemma 3.4$(iii)$]{verachadiimp}. Thus, when such an extension is readily available, as is the case when $\varphi$ is the {\it indicator} function of $C$, hypothesis (H2.4) is omitted.}
\end{enumerate}
\item The function $\varphi$ is globally Lipschitz on $C$ and $\CO^{1}$ on $\inte C$. Moreover, the function $\nabla\varphi$ is globally Lipchitz on $\inte C$.
\item For the sets $C_0$, $C_1$, and $U(\cdot)$  we have:\begin{enumerate}[label=\textbf{H4.\arabic*}:]
\item The set $C_0\subset C$ is nonempty and closed.
\item The graph of $U(\cdot)$ is a $\mathscr{L}\times\mathscr{B}$ measurable set,  and,  for $t\in [0,1],$ $U(t)$ is closed, and bounded uniformly in $t$.
\item The set $C_1\subset\R^n$ is nonempty and closed.
\item The multifunction $U(\cdot)$ is lower semicontinuous.
\item The multifunction $U(\cdot)$ satisfies the constraint qualification (CQ) at $\u$, that is, $$\conv (\bar{N}^L_{U(t)}(\u(t)))\;\hbox{is pointed}\;\,\forall t\in[0,1].\footnote{For more information about the (CQ) property, see \cite[Remark 5.2]{ZNpaper}.}$$
\end{enumerate}
\item There exist $\tilde{\rho}>0$ and $L_g> 0$ such that $g$ is $L_g$-Lipschitz on $\tilde{C}_0(\delta)\times \tilde{C_1}(\delta)$, where \begin{equation*}\label{c01kin} \tilde{C}_i(\delta):=\big[\left(C_i\cap \bar{B}_{\delta}(\x(i))\right)+\tilde{\rho}\bar{B}\big]\cap C,\;\; \text{for}\;\;  i=0,1.\end{equation*}
\end{enumerate}

Let $M_C$ be a bound of the compact set $C$. We denote by $\bar{M}_\psi$  an upper bound of $\|\nabla\psi(\cdot)\|$ on $C$, and  by $2 M_\psi$  a {\it Lipschitz constant} of $\nabla \psi(\cdot)$ over the compact set $C+\frac{\rho}{2}\bar{B}$ chosen large enough so that $M_\psi\geq \frac{4\eta}{\rho}.$  
\begin{remark}\label{assumpf} Since $f$ satisfies (H1), then by \cite[Theorem 1]{hiriartlip} applied to each component, $f_i$, of $f=(f_1,\cdots,f_n)$, 
there exists a function $\tilde{f}\colon[0,1]\times\R^n \times \R^m\f\R^n$, such that, for almost all $t\in[0,1]$, $\tilde{f}(t,\cdot,\cdot)$ is globally Lipschitz on $\R^n \times \R^m$, and $f(t,x,u)=\tilde{f}(t,x,u)$ for all $(x,u)\in \left[C \cap \bar{B}_{\delta}(\x(t))\right]\times \left[\big(U(t)+\tilde{\rho}\bar{B}\big)\cap  \bar{B}_{\delta}(\u(t))\right]$. Moreover, there exists a new constant $M\geq M_{\ell}$ such that $\tilde{f}$ satisfies the assumption (A1) of \cite{verachadiimp}, in which the constant multifunction $U$ is replaced by $U(\cdot)\cap \bar{B}_{\delta}(\u(\cdot))$. Since in this paper we only consider {\it local} optimality notions, then, without loss of generality, we shall use  the function $f$ instead of its extension $\tilde{f}$. In particular, we use that $f$ satisfies the assumption (A1) of \cite{verachadiimp}, and hence, all the results of \cite[Sections 3, 4 \& 5]{verachadiimp} are valid. Now, by \cite[Lemma 3.4]{verachadiimp},  $C$ is $\frac{\eta}{M_{\psi}}$-prox-regular,  and $\varphi$ admits a $\CO^{1}$-extension $\Vphi$ from $C$ to  $\R^n$ satisfying $\partial \varphi(x)=\{\nabla \Vphi(x)\} + N_C(x)$ for all $x\in C$, with     
\begin{equation}\label{normal-cone}
N_C(x)=\{ \lambda \nabla\psi(x): \lambda \ge 0\},\;  \forall x\in \bdry C,
\end{equation}
and for some $K>0$,  
$$\lvert\Vphi(x)\rvert\leq K\;\,\hbox{and}\;\,\|\nabla\Vphi(x)\|\leq K,\;\,\forall x\in\R^n,\;\,\hbox{and}\;\,\|\nabla\Vphi(x)-\nabla\Vphi(y)\|\leq K\|x-y\|,\;\,\forall  x,\,y\in\R^n.$$  
This gives that $(D)$ can be equivalently phrased in terms of the normal cone to $C$ and the extension $\Vphi$ of $\varphi$, as follows 
$$(D) \begin{sqcases}\dot{x}(t)\in  f_{\Vphi}(t,x(t),u(t)) -N_C(x(t)),\;\;\hbox{a.e.}\;t\in[0,1],\\x(0)\in C_0\subset C,	 \end{sqcases}\\$$
where $f_{\Vphi}\colon [0,1]\times \R^n\times \R^m \f \R^n$ is defined by
\begin{equation*} \label{f.phi} f_{\Vphi}(t,x,u):=f(t,x,u)-\nabla\Vphi(x),\;\;\;\forall (t,x,u)\in [0,1]\times \R^n\times \R^m,\end{equation*}
and hence, $\|f_{\Vphi}(t,x,u)\| \le \bar{M}:=M+K$.
\end{remark}

The following notations and facts, extracted from \cite{verachadiimp}, will be used throughout the paper. 
\begin{itemize}
\item For any $(x,u)$ solution of $(D)$, we have by \cite[Equation 50]{verachadiimp} that  
$$\|\dot{x}(t)-f_{\Vphi}(t,x(t),u(t))\| \le \|f_{\Vphi}(t,x(t),u(t))\|\le \bar{M},\; \; t\in[0,1]\;\text{a.e.,}$$ and hence, $\|\dot{x}\|_{\infty}\le 2\bar{M}$.
 \item For given $x(\cdot)\colon [0,1]\to \R^n,$ we define $$I^0(x):=\{t\in [0,1]:  x(t)\in \bdry C\}\;\,\hbox{and}\;\,I^{\-}(x):=[0,1]\setminus I^0(x).$$
\item We define the set $\mathscr{U}$ by
\begin{equation*}\label{scriptU} \mathscr{U}:=\{u:[0,1]\to \R^m : u \;\hbox{is measurable and}\; u(t)\in U(t), \; t\in [0,1] \;\hbox{a.e.}\}.\end{equation*}

\item For $(x,u)\in \W^{1,1}([0,1];\R^n)\times\mathscr{U}$ with $x(0)\in C_0$ and $x(t)\in C$ for all $t\in [0,1]$, we have from \eqref{normal-cone} and Filippov selection theorem (\cite[Theorem 2.3.13]{vinter}) that $x$  is a solution for  $(D)$ corresponding  to the control  $u$   {\it if and only if} there exists  a nonnegative measurable function $\xi$  supported on $ I^{0}(x)$  such that  $(x,u,\xi)$ satisfies
\begin{equation}\label{admissible-Pg} \dot{x}(t)= f_\Vphi(t,x(t),u(t))-\xi(t) \nabla\psi(x(t)),\;\;\;t\in[0,1]\; \textnormal{a.e.} \end{equation}
In this case, the nonnegative function  $\xi$ supported in $I^{0}(x)$ with $(x,u,\xi)$ satisfying equation \eqref{admissible-Pg}, is unique,  belongs to  $L^{\infty}([0,1];\R^{+})$,  and 
\begin{equation}\label{boundxig}  \begin{cases} \xi(t)=0 &\;\;\hbox{for}\;\;t\in I^\-(x),\\[4pt]\xi(t)=\frac{\|\dot{x}(t)-f_{\Vphi}(t,x(t),u(t))\|}{\|\nabla\psi(x(t))\|}\in \left[0, \frac{\bar{M}}{2\eta}\right]&\;\;\hbox{for}\;\;t\in I^0(x) \;\textnormal{a.e.},\\[5pt] \|\xi\|_{\infty}\le \frac{\bar M}{2\eta}.	
\end{cases}
\end{equation}
\item Since $(\x,\u)$ solves the dynamic $(D)$,  we denote by  $\bar{\xi}$  the corresponding function in $ L^{\infty}([0,1];\R^+)$ such that $(\x,\u,\bar{\xi})$ satisfies \eqref{admissible-Pg} and \eqref{boundxig}.
 \item  Let  $(\gk)_k$ be  a sequence  satisfying \begin{equation}  \label{assumpgk} \gk>\frac{2\bar{M}}{\eta}
 \;\,\hbox{for all}\; k\in\N,\;\hbox{and}\; \gk \xrightarrow[k\f\infty ]{}\infty.\end{equation}
 \item The sequence $(\a_k)_k $ is defined by  \begin{equation} \label{velo2} \a_{k}:=\frac{\ln \left(\frac{\eta\gk}{2\bar{M}}\right)}{\gk},\;\;\;k\in\N. \end{equation}
 By \eqref{assumpgk} and \eqref{velo2}, we have that  \begin{equation}\label{velo2bis}  \gk e^{-\a_{k}\gk}=\frac{2\bar{M}}{\eta},\;\;\a_k>0,\;\; \a_k\searrow \;\hbox{and}\;\lim_{k\f\infty}\a_{k}=0.
\end{equation}
\item The sequence  $(\rho_k)_k$ is defined by $\rho_{k}:=\frac{\a_k}{\eta}$ for all $k\in\N$. By (\ref{velo2bis}) we have that  $ \rho_k>0$ for all $k\in\N$, $\rho_k\searrow\,$ and $\,\lim\limits_{k\f\infty}\rho_{k}=0.$
\item For $k\in\N$, the {\it approximating} system $(D_{\gk})$ is defined as
$$({D}_{\gk}) \begin{sqcases} \dot{x}(t)=f_\Vphi(t,x(t),u(t))-\gk e^{\gk\psi(x(t))} \nabla\psi(x(t))\;\;\textnormal{a.e.}\; t\in[0,1],\\ 
x(0)\in C. \end{sqcases}$$
Lemma 4.1 of \cite{verachadiimp} yields that, for each $k$, the system $(D_{\gamma_k})$ with given $x(0)=c_{\gk}\in C$ and $u_{\gk}\in \mathscr{U}$, has a unique solution $x_{\gk} \in \W^{1,2}([0,1];\R^n)$ such that  $x_{\gk}(t)\in C$ for all $t\in [0,1]$, and $(\|\dot{x}_{\gk}\|_2)_k$ is uniformly bounded. To such a solution $x_{\gk}$, we associate
 \begin{equation} \label{defxi} \xi_{\gk}(\cdot):= \gk e^{\gk\psi(x_{\gk}(\cdot))}.\end{equation}
 \item For $k\in\N$, we define the set $C(k):=\{x\in C : \psi(x)\leq -\a_k\}\subset C.$
\item The sequence $(\bar{c}_{\gk})_k$ is defined by \begin{equation*}\label{cbark} \bar{c}_k:= 
     \begin{cases}\x(0),\;\forall k\in\N, &\;\;\hbox{if}\;\x(0)\in\inte C,\vspace{0.1cm}\\ {\x(0)}-\rho_k\frac{\nabla\psi(\x(0))}{\|\nabla\psi(\x(0))\|},\;\forall k\in\N,&\;\;\hbox{if}\; \x(0)\in\bdry C.\end{cases}\end{equation*}
 From \cite[Remark 3.6$(ii)$]{verachadiimp} we have that, for $k$ sufficiently large, $\bar{c}_{\gk}\in \inte C (k)$. Moreover, $\bar{c}_{\gk}\f \x(0)$ as $k\f\infty$.
 
 \item For each $k\in\N$, we denote by $\x_{\gk}$ the unique solution of ($D_{\gk}$) corresponding to ($\bar{c}_{\gk}, \u)$. Theorem 4.1 of \cite{verachadiimp} yields that $\x_{\gk}$  converges in $C$ uniformly to $\x$.
 \item We fix $\delta_{o}>0$ such that \begin{equation*}\label{delta0def} \delta_{o}\leq \begin{cases}\min\{\hat{r}_{\x(0)},\delta\}&\qquad\;\hbox{if}\;\x(0)\in\inte C,\vspace{0.1cm}\\ \min\{r_{\bbp o},\delta\}&\qquad\;\hbox{if}\;\x(0)\in\bdry C,\end{cases} \end{equation*}
where $r_{\bbp o}>0$ is the constant in \cite[Theorem 3.1$(iii)$]{verachadiimp}, and $\hat{r}_{\x(0)}>0$ with $\hat{k}_{\x(0)}\in\N$  are the constants in \cite[Remark 3.6$(ii)$]{verachadiimp} corresponding to  $c:=\x(0)$.
\item We define the set ${C_0}(k)$ by \begin{equation*}\label{c0kdef}{C_0}(k):= 
     \begin{cases}C_0\cap \bar{B}_{\delta_{o}}\bp\left({\x(0)}\right) ,\;\forall k\in\N, &\;\hbox{if}\;\x(0)\in\inte C,\vspace{0.1cm}\\ \left[C_0\cap\bar{B}_{\delta_{o}}\bp\left({\x(0)}\right)\right]-\rho_k\frac{\nabla\psi(\x(0))}{\|\nabla\psi(\x(0))\|},\;\forall k\in\N,&\;\hbox{if}\; \x(0)\in\bdry C.\end{cases}\end{equation*}
     One can easily verify that 
     \begin{equation}\label{limC0} \lim\limits_{k\to\infty} C_0(k) =C_0\cap\bar{B}_{\delta_{o}}\bp({\x(0)})\;\,\hbox{and}\,\;C_0(k) \subset \tilde{C}_0(\delta),\;\,\hbox{for}\,\;k\;\hbox{large enough}.
     \end{equation} Moreover, from \cite[Theorem 3.1$(iii)$ \& Remark 3.6$(ii)$]{verachadiimp}, we have that, for $k$ sufficiently large, 
     \begin{equation} \label{normalc1c0} C_0(k)\subset C(k)\subset C. \end{equation}
\item We define the set ${C_1}(k)$ by       
\begin{equation*} \label{C1kdefinition} C_1(k):=\left[\left(C_1\cap\bar{B}_{\delta_{o}}\bp({\x(1)})\right) -\x(1)+{\x}_{\gk}(1)\right]\cap C,\;\;\;k\in\N.\end{equation*}
One can easily verify that \begin{equation}\label{limC1} \lim_{k\to\infty} C_1(k) =C\cap C_1\cap\bar{B}_{\delta_{o}}\bp({\x(1)})\;\,\hbox{and}\,\;C_1(k) \subset \tilde{C}_1(\delta)\;\,\hbox{for}\,\;k\;\hbox{large enough}.\end{equation} 

\end{itemize}

\subsection{Statement of the main result} Before presenting the main result of this paper, namely the necessary optimality conditions for a  given $\C\t\W^{1,2}$-local minimizer  of $(P)$, we establish the following existence of optimal solution theorem for $(P)$, which is parallel to \cite[Theorems 4.1]{cmo0,cmo2} where a discretization technique is used. 
\begin{theorem}[Existence of optimal solution for $({P})$]\label{exop} Assume that  \textnormal{(H2)-(H4.3)} hold, and that$\sp:$
\begin{itemize}
	\item For fixed $(x,u)\in C\times \mathbb{U} $, $f(\cdot, x,u)$ is Lebesgue-measurable; and there exists $ M>0$ such that for \textnormal{a.e.} $t \in [0, 1]$ we have: $(x, u)\mapsto f (t, x, u)$ is continuous on $C\times U(t)$; for all $u\in U(t)$, $x\mapsto f (t, x, u)$ is $M$-Lipschitz on $C$; and  $\|f(t,x,u)\| \leq M $  for all $(x,u)\in  C \times U(t).$ 
	\item The function $g\colon \R^{n}\times\R^n\to \R\cup\{\infty\}$ is lower semicontinuous.
	\item A minimizing sequence $(x_j,u_j)$ for $({P})$ exists such that $(\|\dot{u}_j\|_2)_j$ is bounded.
	\item The problem $(P)$ has at least one admissible pair $(y_o,v_o)$ with $(y_o(0), y_o(1))\in \dom g$.
\end{itemize}
Then the problem $({P})$ admits a global optimal solution $(\tilde{x},\tilde{u})$ such that, along a subsequence, we have  
$${x}_{j}\xrightarrow[{C([0,1]; \R^n)}]{\textnormal{uniformly}}{\tilde{x}},\;\;{u}_{j}\xrightarrow[{C([0,1]; \R^m)}]{\textnormal{uniformly}}{\tilde{u}},\;\;\;\dot{x}_{j}\xrightarrow[{L^{\infty}([0,1]; \R^n)}]{\textnormal{weakly*}}\dot{\tilde{x}},\;\; \hbox{and}\;\;\dot{u}_{j}\xrightarrow[{L^2([0,1]; \R^m)}]{\textnormal{weakly}}\dot{\tilde{u}}.$$
\end{theorem}   
\begin{proof} \sloppy The fact that $(P)$ admits an admissible  pair  $(y_o,v_o)$ with $(y_o(0), y_o(1))\in \dom g$ yields that $\inf_{(x,u)} ({P})<\infty.$  Since all admissible solutions $(x,u)$ of ($P$)  have  $(x(0),x(1))$ in the compact set $C_0\times (C_1\cap C)$, then the  lower semicontinuity of $g$  gives  that $\inf_{(x,u)} ({P})$ is finite. As the minimizing sequence $(x_{j},u_{j})_j$  solves  $(D)$ and $C$ is compact,  it follows that the sequences  $(\|\dot{x}_j\|_{\infty})_j$   and $(\|x_j\|_{\infty})_j$ are bounded by $2\bar{M}$ and $M_C$, respectively. On the other hand, by hypothesis, we have   $(\|\dot{u}_j \|_2)_j$  is bounded, and by the $t$-uniform boundedness of $U(t)$  in (H4.2),  the sequence $(\|{u}_j\|_{\infty})_j$ is also bounded. Hence,  Arzela-Ascoli's theorem produces  a subsequence, we do not relabel, of  $(x_j , u_j )_j$, that converges uniformly to an absolutely continuous pair $(\tilde{x},\tilde{u})$ with $(\tilde{x}(t),\tilde{u}(t))\in C\times U(t)$ for all $t\in [0,1]$,  $(\dot{x}_{j})_j$ converging to $\dot{\tilde{x}}$ in the weak*-topology of $L^{\infty}$, and $(\dot{u}_{j})_j$ converging weakly in $L^2$ to $ \dot{\tilde{u}}$. As for all $j\in\N$,  $(x_j(0), x_j(1))\in  C_0\times C_1$, then (H4.1) and (H4.3) yield that $(\tilde{x}(0), \tilde{x}(1))\in C_0\times C_1$.  
To prove that $(\tilde{x},\tilde{u})$  satisfies the sweeping process in $(D)$, we shall use the equivalence  invoking \eqref{admissible-Pg}. Let $(\xi_j)_j$ be the $L^{\infty}([0,1],\R^{+})$- sequence  associated via \eqref{admissible-Pg}-\eqref{boundxig}  to the  sequence $(x_j,u_j)_j$ admissible for  $(D)$. As  \eqref{boundxig} yields that $(\|\xi_j\|_{\infty})_j$ is bounded by $\frac{\bar{M}}{2\eta}$, then $(\xi_j)_j$ admits a subsequence, we do not relabel, that  weakly* converges in $L^{\infty}([0,1];\R^+)$ to some $\tilde{\xi} \in L^{\infty}([0,1];\R^+)$. 
 Using  that $(x_j,u_j,\xi_j)$ satisfies  \eqref{admissible-Pg} together  with the assumptions on $f$, and the properties of $\Vphi$  and $\psi$,  it easily follows upon  taking the limit as $ j\f \infty$  in 
 $$x_{j}(t)= x_{j}(0) + \int_{0}^t [f_{\Vphi}(s,x_{j}(s),u_{j}(s)) -\xi_{j}(s) \nabla\psi(x_{j}(s))]\; ds, \;\; \forall t\in[0,1],$$
 that also  $(\tilde{x},\tilde{u},\tilde{\xi})$   satisfies \eqref{admissible-Pg}. We now show that $\tilde{\xi}$ is supported in $ I^{0}(\tilde{x})$. Let $t\in I^{\-}(\tilde{x})$ be fixed,  that is, $\tilde{x}(t)\in \inte C$. Since $(x_j)_j$ converges uniformly to $\tilde{x}$, then we can find $\tilde{\delta} >0$ and $j_o \in \N$ such that, for all $s\in (t-\tilde{\delta},t+\tilde{\delta})\cap [0,1]$ and for all $j\ge j_o$, we have $ x_j(s)\in \inte C$, and hence $\xi_j(s)=0$, as $\xi_j$ satisfies \eqref{boundxig}. Thus, $\xi_j(s)\f 0$ for $ s\in (t-\delta_o,t+\delta_o)\cap [0,1]$, and whence, $\tilde{\xi}(t)=0$, proving that $\tilde{\xi}$ is supported in $I^{0}(x)$. Thus, $(\tilde{x}, \tilde{u})$ solves ($D$) and  $(\tilde{x},\tilde{u},\tilde{\xi})$ satisfies \eqref{boundxig}. Therefore, $(\tilde{x},\tilde{u})$ is admissible for $({P})$. Owed to the lower semicontinuity of $g$  and  to $(\tilde{x},\tilde{u})$ being the uniform limit of the minimizing sequence $(x_{j},u_{j})_j$, the optimality of $(\tilde{x},\tilde{u})$  for $({P})$  follows readily. \end{proof}

The following theorem, Theorem \ref{mpw12bis},  is the main result of this paper. It provides necessary optimality conditions in the form of weak Pontryagin principle for a $\C\t\W^{1,2}$-local minimizer $(\x,\u)$ in $(P)$. These optimality conditions extend  those given in \cite[Theorem 4.8]{ZNpaper}, where $(\x,\u)$ is a  $\W^{1,2}\t \W^{1,2}$-local minimizer and the perturbation function is {\it autonomous}. Note that in the statement of  Theorem \ref{mpw12bis}, we use the following {\it nonstandard} notions of subdifferentials, which are {\it strictly smaller} than their counterparts in standard notions:
\begin{itemize}
\item $\partial_\ell\varphi$ and $\partial^2_\ell\varphi$  defined on $C$, are   the {\it extended  Clarke generalized gradient} and the {\it extended  Clarke generalized Hessian} of $\varphi$, respectively (see \cite [Equation (9)-(10)]{verachadiimp}). Note that if $\partial_\ell\varphi(x)$ is a singleton, then we use the notation $\nabla_{\bp\ell}$ instead of  $\partial_\ell$.
\item $\partial^{\sp (x,u)}_\ell f (t,\cdot,\cdot)$ is the {\it extended  Clarke generalized Jacobian} of $f(t,\cdot,\cdot)$ defined on the set     
 $\left[C \cap \bar{B}_{\delta}(\x(t))\right]\times \big[(U(t)+\tilde{\rho}\bar{B})\cap  \bar{B}_{\delta}(\u(t))\big]$ (see \cite [Equation (12)]{verachadiimp}).
\item $\partial^2_\ell\psi$ is the {\it Clarke generalized Hessian relative} to $\inte C$ of $\psi$ (see \cite [Equation (11)]{verachadiimp}).
\item $\partial^L_\ell g$ is the {\it limiting subdifferential} of $g$ {\it relative} to  $\inte \big(\tilde{C}_0(\delta)\times \tilde{C}_1(\delta)\big)$ (see \cite [Equation (8)]{verachadiimp}).
\end{itemize}

\begin{theorem}[Necessary optimality conditions]\label{mpw12bis} Let $(\x,\u)$ be a $\C\t\W^{1,2}$-local minimizer for  $(P)$ with associated $\delta>0$ at which \textnormal{(H1)-(H5)} hold. Then, 
there exist $\l\geq 0$, an adjoint vector ${\p}\in BV([0,1];\R^n)$, a finite signed Radon measure $\nnu$ on $[0,1]$ supported on $I^{0}(\x)$,  $L^{\infty}$-functions $\zzeta(\cdot)$, $\ttheta(\cdot)$ and $\vvartheta(\cdot)$ in $\mathscr{M}_{n\times n}([0,1])$, and an $L^{\infty}$-function $\oomega(\cdot)$ in $\mathscr{M}_{n\times m}([0,1])$,  such that 
$$\left((\zzeta(t),\oomega(t)),\ttheta(t), \vvartheta(t)\right)\bbp\in\bbp \;\partial^{\sp (x,u)}_\ell f(t,\x(t),\u(t))\times \partial^2_\ell\varphi(\x(t))\times \partial^2_\ell\psi(\x(t)),\;\sp t\in [0,1]\; \textnormal{a.e.},$$
and the following holds\sp$:$
\begin{enumerate}[$(i)$]
\item {\bf(The admissible equation)}
\begin{enumerate}[$(a)$]
\item $\dot{\x}(t)= f(t,\x(t),\u(t))- \nabla_{\bp\ell}\sp\sp\varphi(\x(t))-\bar{\xi}(t) \nabla\psi(\x(t)),\;\;  t\in [0, 1]\; \textnormal{a.e.},$
\item $\psi(\x(t))\le 0,\;\;\forall\sp t\in [0,1];$
\end{enumerate}
\item {\bf (The nontriviality condition)} $$\|\p(1)\| + \l=1;$$
\item {\bf (The adjoint equation)} 
 For any $h\in C([0,1];\R^n),$ we have\vspace{0.1cm} \begin{eqnarray*}\hspace*{0.9cm} \int_{[0,1]}\<h(t),d\p(t)\>&= &\int_0^1 \left\<h(t),\left(\ttheta(t)-\zzeta(t)\tran\right) \p(t)\right\>\sp dt \\[3pt]&+& \int_0^1 \xxi(t)\left\<h(t),\vvartheta(t)p(t)\right\>\sp dt+  \int_{[0,1]} \<h(t),\nabla \psi(\x(t))\>\sp d\nnu;\\[-0.3cm]\end{eqnarray*} 
 \item {\bf (The complementary slackness conditions)} 
\begin{enumerate}[$(a)$]
\item $\bar{\xi}(t)=0,\;\;\forall\sp t\in I^{\-}(\x)$,
\item $\bar{\xi}(t)\<\nabla\psi(\x(t),\bar{p}(t)\>=0,\;\;\forall\sp t\in [0, 1]\; \textnormal{a.e.};$
\end{enumerate}
\item {\bf (The transversality equation)} $$ (\p(0),-\p(1))\in  \l\partial_\ell^L g(\x(0),\x(1))+ \big[N_{C_0}^L(\x(0))\times\;N^L_{C_1}(\x(1))\big];$$
\item {\bf (The weak maximization condition)}  $$\oomega(t)\tran \p(t)\in \conv\bar{N}^L_{U(t)\cap \bar{B}_{\delta}(\u(t))}(\u(t)),\;\;  t\in [0,1]\; \textnormal{a.e.}\vspace{0.1cm}$$
If in addition there exist $\e_{o}>0$ and $r>0$ such that $U(t)\cap \bar{B}_{\e_{o}}(\u(t))$ is $r$-prox-regular for all $t\in [0,1]$, then for $t\in [0,1]$ \textnormal{a.e.}\sp, we have
$$\textstyle \max\left\{\left\<\oomega(t)\tran \p(t),u\right\>- \frac{\|\oomega(t)\tran \p(t)\|}{\min\{\e_{o},2r\}}\|u-\u(t)\|^2 : u\in U(t)\right\}\;\hbox{is attained at}\;\bar{u}(t).$$ 
\end{enumerate}
Furthermore, if $C_1 = \R^n$, then $\lambda\not= 0$ and is taken to be $1$, and the nontriviality condition $(i)$ is eliminated.
\end{theorem}

 \begin{remark} \label{specialcases} The following are simplified versions of  the weak maximization condition of  Theorem \ref{mpw12bis} for the special cases: $(a)$ $U(t)$ is $r$-prox-regular for all $t\in [0,1]$, $(b)$ $U(t)\cap \bar{B}_{\e_{o}}(\u(t))$ is convex for all $t\in [0,1]$, and $(c)$ $U(t)$ is convex for all $t\in [0,1]$.
\begin{enumerate}[($a$)]
\item We take $\e_{o}\f\infty$ to get that for $t\in [0,1]$ \textnormal{a.e.}\sp,
$$\textstyle  \max\left\{\left\<\oomega(t)\tran \p(t),u\right\>- \frac{\|\oomega(t)\tran \p(t)\|}{2r}\|u-\u(t)\|^2 : u\in U(t)\right\}\;\hbox{is attained at}\;\bar{u}(t).$$ 
\item We take $r\f\infty$ to get that for $t\in [0,1]$ \textnormal{a.e.}\sp,
$$ \textstyle \max\left\{\left\<\oomega(t)\tran \p(t),u\right\>- \frac{\|\oomega(t)\tran \p(t)\|}{\e_{o}}\|u-\u(t)\|^2 : u\in U(t)\right\}\;\hbox{is attained at}\;\bar{u}(t).$$ 
\item We take both $\e_{o}\f\infty$ and $r\f\infty$ to get that for $t\in [0,1]$ \textnormal{a.e.}\sp,   
\begin{equation*}\label{max}\textstyle  \max\left\{\left\<\oomega(t)\tran \p(t),u\right\> : u\in U(t) \right\}\;\hbox{is attained at}\;\bar{u}(t).\end{equation*}
\end{enumerate}
\end{remark}

\section{Proof of the main result} \label{proofmn} The proof of Theorem \ref{mpw12bis} is presented in three steps.
\subsection*{Step 1: Approximating problems for $(P)$}
 We introduce the following sequence of approximating problems: 
$$ \begin{array}{l} (P_{\gk})\colon\; \textnormal{Minimize}\\[2pt] \hspace{0.65cm}J(x,z,u):= g(x(0),x(1))+\frac{1}{2}\left(\|u(0)-\u(0)\|^2 + z(1)+\|x(0)-\x(0)\|^2\right)\\[5pt] \hspace{0.65cm}  \textnormal{over}\;(x,z,u)\in \W^{1,2}([0,1];\R^n)\times\W^{1,1}([0,1];\R)\times \mathscr{W} \;\textnormal{such that} \\[2.6pt] \hspace{1.15cm}\begin{cases} (D_{\gk}) \begin{sqcases}\dot{x}(t)={f_{\Vphi}}(t,x(t),u(t))-\gk e^{\gk\psi(x(t))} \nabla\psi(x(t)),\;\; t\in[0,1]\;\textnormal{a.e.},\\ \dot{z}(t)=\|\dot{u}(t)-\dot{\u}(t)\|^2,\;\;t\in[0,1]\;\hbox{a.e.},\\ (x(0),z(0))\in {C}_0(k)\times \{0\}, \end{sqcases}\vspace{0.1cm}\\x(t)\in \bar{B}_{\delta}(\x(t))\;\hbox{and}\; u(t)\in U(t)\cap \bar{B}_{\delta}(\u(t)),\;\;\forall t\in[0,1],\vspace{0.1cm}\\(x(1),z(1))\in {C}_1(k)\times [-\delta,\delta].\end{cases} \end{array}$$

\begin{lemma}\label{lasthopefully} For $k$ sufficiently large, the problem $(P_{\gk})$ has an optimal solution $(x_{\gk},z_{\gk},u_{\gk})$ such that,  for  $\xi_{\gk}$ defined in \eqref{defxi}, we have,  along a subsequence, we do not relabel, that
$$u_{\gk}\xrightarrow[{\mathscr{W}}]{\textnormal{strongly}}\u,\;\;\;{x}_{\gk}\xrightarrow[{C([0,1]; \R^n)}]{\textnormal{uniformly}}{\x},\;\;\;z_{\gk}\xrightarrow[{W^{1,1}([0,1];\R^+)}]{\textnormal{strongly}} 0,\;\,\hbox{and}\;\,(\dot{x}_{\gk},\xi_{\gk})\xrightarrow[{L^{\infty}([0,1]; \R^n\times \R^{+})}]{\textnormal{weakly*}}(\dot{\x},\bar{\xi}).$$  
In addition, we have\sp$:$\begin{enumerate}[$(i)$]
\item $x_{\gk}(t)\in C(k)\subset \inte C,\;\forall t\in[0,1]$.
\item $ 0\le \xi_{\gk}(t)\le \frac{2\bar{M}}{\eta},\;\forall t\in [0,1]$.
\item $ \|\dot{x}_{\gk}(t)\|\leq \bar{M}+\frac{2\bar{M}\bar{M}_\psi}{\eta},\;\forall t\in [0,1]$ \textnormal{a.e.}
\item $x_{\gk}(i)\in \big[\left(C_i\cap \bar{B}_{\delta_{o}}(\x(i))\right)+\tilde{\rho}{B}\big]\cap(\inte C)\subset \inte \tilde{C}_i(\delta),\;$ for $i=0,1.$
\end{enumerate}
\end{lemma}
\begin{proof} By \eqref{limC0} and \eqref{limC1}, let  $k$ be large enough so that $C_0(k)\subset \tilde{C}_0(\delta)$ and $C_1(k)\subset \tilde{C}_1(\delta)$. Since  $\x_{\gk}\f\x$ uniformly,  then, for $k$ sufficiently large,  $\x_{\gk}(t)\in \bar{B}_{\delta}(\x(t))$, $\forall t\in [0,1].$ Thus, using that $\bar{c}_{\gk}\in C_0(k)$, for all $k\in \N$,  and $\x(1)\in C_1\cap \bar{B}_{\delta_{o}}(\x(1))$,  it follows that, for $k$ large,  $ (\x_{\gk},\z_{\gk}:=0,\u)$ is an admissible triplet for $(P_{\gk})$. 

Now, fix $k$ large enough so that $C_0(k)\times C_1(k) \subset \tilde{C}_0(\delta) \times \tilde{C}_1(\delta)$ and $(\x_{\gk},0,\u)$ is admissible for $(P_{\gk})$. Using (H5) and the definition of $J(x,z,u)$, we obtain that  $J(x,z,u)$ is bounded from below, and  hence, $\inf (P_{\gk})$ is {\it finite}. Let $(x_{\gk}^n,z^n_{\gk},u^n_{\gk})_n\in \W^{1,2}([0,1];\R^n)\times\W^{1,1}([0,1];\R)\times \mathscr{W} $ be a minimizing sequence for $(P_{\gk})$, that is,  for each $n\in \N$, $(x_{\gk}^n,z^n_{\gk},u^n_{\gk})$ is admissible for $(P_{\gk})$, and 
\begin{equation} \label{chadin1} \lim_{n\f\infty} J(x_{\gk}^n,z^n_{\gk},u^n_{\gk})= \inf (P_{\gk})<\infty.
\end{equation}
Since for each $n$, $x_{\gk}^n$ solves $(D_{\gk})$ for $(x_{\gk}^n(0),u^n_{\gk})$, and $(x_{\gk}^n(0))_n \in C_0(k)\subset C$, then, by \cite[Lemma 4.1]{verachadiimp}, we have that the sequence $(x_{\gk}^n)_n$ is uniformly bounded in $\C([0,1];\R^n)$ and the sequence $(\dot{x}_{\gk}^n)_n$ is uniformly bounded in $L^2$. On the other hand,  from (H4.2), we have that  the sets $U(t)$ are compact and uniformly bounded, then, the sequence $(u^n_{\gk})_n$, which is in $\mathscr{W}$, is uniformly bounded in $\C([0,1];\R^m)$. Moreover, its derivative  sequence, $(\dot{u}^n_{\gk})_n$, must be uniformly bounded in $L^2$, since we have 
 $z^n_{\gk}(t)= \int_0^t \|\dot{u}^n_{\gk}(\tau)-\dot{\u}(\tau)\|^2\sp d\tau, \;\; \forall t\in[0,1],$
 and hence, 
 \begin{equation}\label{dotu-bdd}\|\dot{u}^n_{\gk}\|_2  \le  \|\dot{u}^n_{\gk}-\dot{\u}\|_2 +\|\dot{\u}\|_2= (z_{\gk}^n(1))^{\frac{1}{2}}+ \|\dot{\u}\|_2\le \sqrt{\delta} + \|\dot{\u}\|_2.
 \end{equation}
Therefore, by Arzel\`a-Ascoli theorem, along a subsequence (we do not relabel), the sequence $(x_{\gk}^n, u_{\gk}^n)_n$ converges  {\it uniformly} to a pair $(x_{\gk}, u_{\gk})$  and the sequence $(\dot{x}_{\gk}^n, \dot{u}_{\gk}^n)_n$ converges {\it weakly} in $L^2$ to the pair $(\dot{x}_{\gk},\dot{u}_{\gk})$. Hence, $(x_{\gk}, u_{\gk})\in  \W^{1,2}([0,1];\R^n)\times \mathscr{W}$. Moreover, we have \begin{equation}\label{uw12xw12} \|\dot{u}_{\gk}-\dot{\u}\|_2^2\leq \liminf\limits_{n\f \infty} \|\dot{u}^n_{\gk}-\dot{\u}\|_2^2.\end{equation}
We define 
\begin{equation}\label{zgk}z_{\gk}(t):= \int_0^t \|\dot{u}_{\gk}(\tau)-\dot{\u}(\tau)\|^2\sp d\tau, \;\; \forall t\in[0,1].
\end{equation}
We claim that $(x_{\gk},z_{\gk},u_{\gk})$ is optimal for $({P}_{\gk})$. First we prove its admissibility. Clearly we have $$z_{\gk}\in \W^{1,1}([0,1];\R),\;\;\dot{z}_{\gk}(t)=\|\dot{u}_{\gk}(t)-\dot{\u}(t)\|^2\;\;\forall t\in[0,1] \hbox{ a.e.,\;and} \;\; z_{\gk}(0)=0.$$
Moreover, since $\|\dot{u}^n_{\gk}-\dot{\u}\|_2^2= z^n_{\gk}(1)\in [-\delta,\delta]$, \eqref{uw12xw12} yields that \begin{equation} \label{ykandzk1}z_{\gk}(1)\in [-\delta,\delta].\end{equation}
The inclusions $(x_{\gk}(0), x_{\gk}(1))\in C_0(k)\times C_1(k)$, and  $x_{\gk}(t)\in \bar{B}_{\delta}(\x(t))$  and $u_{\gk}(t)\in U(t)\cap \bar{B}_{\delta}(\u(t))$, for all $t\in[0,1]$, follow directly from $C_0(k)$, $C_1(k)$, $\bar{B}_{\delta}(\x(t))$ and $U(t)\cap \bar{B}_{\delta}(\u(t))$ being closed for all $t\in[0,1]$, and from the uniform convergence, as $n\f\infty$, of the sequence $(x_{\gk}^n, u_{\gk}^n)$ to $(x_{\gk}, u_{\gk})$.  To prove that  $x_{\gk}$ is the solution of  $(D_{\gk})$ corresponding to  $(x_{\gk}(0),u_{\gk})$,     take the limit, as $n\f\infty$, in this integral form of the admissible equation in  $(D_{\gk})$ for  $(x_{\gk}^n, u^n_{\gk})$,
$$ {x}_{\gk}^n(t) =  x_{\gk}^n(0)+\int_0^t \left[{f}_\Vphi(s,x_{\gk}^n(s),u^n_{\gk}(s))- \gk e^{\gk\psi(x_{\gk}^n(s))} \nabla\psi(x_{\gk}^n(s))\right]\,ds,\;\; \forall \; t\in[0,1], $$
and use  that  $(x_{\gk}^n(t), u_{\gk}^n(t))\in [C\cap \bar{B}_{\delta}(\x(t))]\times [U(t)\cap \bar{B}_{\delta}(\u(t))]$,   $(x_{\gk}^n, u^n_{\gk})$  converges uniformly to  $(x_{\gk}, u_{\gk})$,  (H1) and (H2.1) hold,   and that $\Vphi$ is $\CO^{1}$, we conclude  that $(x_{\gk},u_{\gk})$  satisfies the same equation, that is,
 $$\dot{x}_{\gk}(t)= {f}_\Vphi(t,x_{\gk}(t),u_{\gk}(t))- \gk e^{\gk\psi(x_{\gk}(t))} \nabla\psi(x_{\gk}(t)),\;\;t\in[0,1]\;\textnormal{a.e.}$$
This terminated the proof of the admissibility of $(x_{\gk},z_{\gk},u_{\gk})$ for $({P}_{\gk})$. For its optimality, use (\ref{chadin1}) and the uniform convergence of $(x_{\gk}^n, u_{\gk}^n)$ to $(x_{\gk}, u_{\gk})$, it follows that \begin{eqnarray*}\inf (P_{\gk}) &=& \lim_{n\f\infty} J(x_{\gk}^n,z^n_{\gk},u^n_{\gk})\\[3pt] & =& \lim_{n\f\infty}\left(\bbp g(x^n_{\gk}(0),x^n_{\gk}(1))\bbp +\bbp \frac{1}{2}\left(\|u^{n}_{\gk}(0)-\u(0)\|^2 \bbp +\bbp  \|\dot{u}_{\gk}^n-\dot{\u}\|_2^2\bbp +\bbp \|x^n_{\gk}(0)-\x(0)\|^2\right)\bp\right)\\[3pt]&=& g(x_{\gk}(0),x_{\gk}(1))\bbp+\bbp\frac{1}{2}\|u_{\gk}(0)-\u(0)\|^2 \bbp+\bbp\frac{1}{2}\liminf_{n\f\infty}\|\dot{u}_{\gk}^n\bbp-\bbp\dot{\u}\|_2^2 \bbp+\bbp\frac{1}{2}\|x_{\gk}(0)\bbp-\bbp\x(0)\|^2\\[3pt] &\geq& g(x_{\gk}(0),x_{\gk}(1))+\frac{1}{2}\|u_{\gk}(0)-\u(0)\|^2 +\frac{1}{2}\|\dot{u}_{\gk}-\dot{\u}\|_2^2 +\frac{1}{2}\|x_{\gk}(0)-\x(0)\|^2\\[3pt] &=& J(x_{\gk},z_{\gk},u_{\gk}).
 \end{eqnarray*}
Therefore, for each such  $k$, $(x_{\gk},z_{\gk},u_{\gk})$ is an optimal solution for $(P_{\gk})$.

For the convergence of $(x_{\gk},z_{\gk},u_{\gk})_k$ when $k\f \infty$, we  first note that the sequence  $(u_{\gk})_k$ in $\mathscr{W}$  has uniformly bounded derivative in $L^2$. This follows  from using \eqref{zgk} and \eqref{ykandzk1} to obtain that \eqref{dotu-bdd} also holds when   $(u^n_{\gk},z^n_{\gk})$ is replaced by $(u_{\gk},z_{\gk})$. Hence, using arguments similar to those used above for the sequence $(u_{\gk}^n)_n$, we obtain the existence of $u\in \mathscr{W}$ such that, along a subsequence not relabled, $u_{\gk}$ converges uniformly to $u$, $\dot{u}_{\gk}$ converges weakly in $L^2$ to $\dot{u}$, and \begin{equation}\label{uw12bis} \|\dot{u}-\dot{\u}\|_2^2\leq \liminf_{k\f \infty} \|\dot{u}_{\gk}-\dot{\u}\|_2^2. \end{equation}
On the other hand, by \eqref{normalc1c0} we have $C_0(k)\subset C$, for $k$ large, then,  by \cite[Theorem 4.1 \& Lemma 4.2]{verachadiimp}, the sequence $(x_{\gk}, \xi_{\gk})_k$, where $\xi_{\gk}$ is given via \eqref{defxi},  admits  a subsequence, not relabled, such that $(x_{\gk})_k$ converges uniformly to some $x\in \W^{1,2}([0,1];\R^n)$  with images in $C$,  $(\dot{x}_{\gk}, \xi_{\gk})_k$ converges weakly in $L^2$ to $(\dot{x}, \xi)$, and  $\xi$ is supported on $I^0(x)$. Furthermore, $(x, u,\xi)$ satisfies \eqref{admissible-Pg}-\eqref{boundxig} and $x$  uniquely  solves   $(D)$ for $(x(0),u)$. Now, as
 $x_{\gk}(0)\in C_0(k)$, for $k$ large, equation \eqref{limC0}(a), implies  that $x(0)\in C_0\cap \bar{B}_{\delta_o}(\x(0))$. Hence,  $(x,u,\xi)$ satisfies
$$\begin{cases} \dot{x}(t)= f_\Vphi(t,x(t),u(t))-\xi(t) \nabla\psi(x(t))\in{f}(t,x(t),u(t))- \partial\varphi (x(t))\;\;\textnormal{a.e.}\; t\in[0,1],\\x(0)\in C_0\cap \bar{B}_{\delta_{o}}(\x(0)). \end{cases} $$
Since $x_{\gk}(1)\in C_1(k) $ for $k$ large, equation \eqref{limC1}(a)  implies that $x(1)\in  C_1\cap \bar{B}_{\delta_{o}}(\x(0))$. Furthermore, having, for all $t\in[0,1]$, that $x_{\gk}(t)\in \bar{B}_{\delta}(\x(t))$ and  $u_{\gk}(t)\in U(t)\cap \bar{B}_{\delta}(\u(t))$,  and that  $(x_{\gk},u_{\gk})$ converges uniformly to $(x,u)$, we get $x(t)\in \bar{B}_{\delta}(\x(t))$ and  $u(t)\in U(t)\cap \bar{B}_{\delta}(\u(t))$, for all  $t\in[0,1]$. In addition, we have that $$\|\dot{u}-\dot{\u}\|^2_2\overset{\eqref{uw12bis}}{\leq}\liminf_{k\f\infty}\|\dot{u}_{\gk}-\dot{\u}\|^2_2=\liminf_{k\f\infty}z_{\gk}(1)\overset{\eqref{ykandzk1}}{\in}[-\delta,\delta],$$
proving  that $(x,u)$ is admissible for $({P})$. Thus,  by the local optimality of $(\x,\u)$ for $({P})$, we have that
\begin{equation}\label{mp2} g(\x(0),\x(1))\leq g(x(0),x(1)).	
 \end{equation} 
Now by   using the admissibility of $(\x_{\gk},0,\u)$ and the optimality of $(x_{\gk},z_{\gk},u_{\gk})$ for $({P}_{\gk})$, it follows that \begin{equation}\label{newimport2} J(x_{\gk},z_{\gk},u_{\gk})\leq g(\x_{\gk}(0),\x_{\gk}(1))+\frac{1}{2}\|\x_{\gk}(0)-\x(0)\|^2.\end{equation}
Hence, using the uniform convergence of $\x_{\gk}$ to $\x$, (\ref{newimport2}), (\ref{mp2}), the Lipschitz continuity of $g$,  and  the uniform convergence of $x_{\gk}$ to $x$, we get 
\begin{eqnarray*} g(x(0),x(1)) \hspace{-0.25cm} &\leq & \hspace{-0.25cm} \liminf_{k\f \infty}\left(\bbp g(x_{\gk}(0),x_{\gk}(1))\bbp+\bbp\frac{1}{2}\left(\bbp\|u_{\gk}(0)-\u(0)\|^2\bbp +\bbp \|\dot{u}_{\gk}-\dot{\u}\|_2^2\bbp+\bbp\|x_{\gk}(0)-\x(0)\|^2\bbp\right)\bp\bbp\right)\bp \\& =& \hspace{-0.25cm} \liminf_{k\f \infty} J(x_{\gk},z_{
\gk},u_{\gk}) \\&\leq  &\hspace{-0.25cm} \liminf_{k\f \infty} \left(g(\x_{\gk}(0),\x_{\gk}(1))+\frac{1}{2}\|\x_{\gk}(0)-\x(0)\|^2\right) \bp =\bbp g(\x(0),\x(1)) \leq g(x(0),x(1)).\end{eqnarray*} 
Thus \begin{equation} \label{sobolev} u(0)=\u(0)  \;\; \text{and}\;\; \liminf_{k\f \infty}\left(\ \|\dot{u}_{\gk}-\dot{\u}\|_2^2\right)=0,\;\;\hbox{and}\end{equation}  \begin{equation}\label{sobolevbis} x(0)=\x(0)\;\;\hbox{and}\;\;g(\x(0),\x(1))=g(x(0),x(1)).\end{equation}
The equality (\ref{sobolev}) gives the existence of a subsequence of $u_{\gk}$, we do not relabel, such that  $\dot{u}_{\gk}$ converges {\it strongly }in $L^2$ to $\dot{\u}$. It results that $u_{\gk}$ converges uniformly to $\u$, and hence, $u=\u$. Consequently, $$u_{\gk}\xrightarrow[{\mathscr{W}}]{\textnormal{strongly}}\u.$$
This yields that $z_{\gk}\f 0 $ in the strong topology of ${W^{1,1}([0,1];\R^+)}$. Moreover, as $u=\u$, then the functions   $x$ and $\x$ solve the dynamic $({D})$ with the same control $\bar{u}$ and  initial condition, see (\ref{sobolevbis}), hence,   by the uniqueness of the solution of $({D})$, we have $x=\x$. Using \eqref{boundxig}, we obtain that also $\xi=\bar{\xi}$. Therefore, $${x}_{\gk}\xrightarrow[{C([0,1]; \R^n)}]{\textnormal{uniformly}}{\x}\;\;\hbox{and}\;\;(\dot{x}_{\gk}, \xi_{\gk})\xrightarrow[{L^2([0,1]; \R^n\times \R^{+})}]{\textnormal{weakly}}(\dot{\x},\bar{\xi}).$$ 
 
 As $x_{\gk}(0)\in C_0(k)$, we have that $x_{\gk}(0)\in C(k)$ for $k$ sufficiently large. Hence using \cite[Theorem 5.1]{verachadiimp}, we obtain that the conditions $(i)$-$(iii)$ of the \enquote{In addition} part, hold true. This implies that a subsequence, we do not relabel, of $(\dot{x}_{\gk}, \xi_{\gk})$ also converges  weakly* in $L^{\infty}([0,1], \R^{n+1})$ to $(\dot{\x},\bar{\xi})$.   
 Moreover, since $x_{\gk}(1)\in\left[\left(C_1\cap \bar{B}_{\delta_{o}}(\x(1))\right)-\x(1)+\x_{\gk}(1)\right]\cap (\inte C)$ and $\x_{\gk}(1)$ converges to $ \x(1)$,  it follows that  $x_{\gk}(1)\in \big[\left(C_1\cap \bar{B}_{\delta_{\bbbp o}}(\x(1))\right)+\tilde{\rho}{B}\big]\cap(\inte C)$, for $k$ sufficiently large. On the other hand, the definition of $C_0(k)$ and the convergence of  $\rho_{k}$ to  $0$ yield that, for $k$  large enough,  $x_{\gk}(0)\in \big[\left(C_0\cap \bar{B}_{\delta_{o}}(\x(0))\right)+\tilde{\rho}{B}\big]\cap(\inte C)$. \end{proof}

\subsection*{Step 2: Maximum principal for the approximation problems} We proceed and we  rewrite the approximating problems $(P_{\gk})$ as a {\it standard} optimal control problem with state constraints in which   the control $u$, which is in $\W^{1,2}$, is considered as  another state variable and its derivative, $v:=\dot{u}$ is the control. For $\bar{v}:=\dot{\u}$, the problem $(P_{\gk})$ is reformulated in the following manner:
$$ \begin{array}{l} (P_{\gk})\colon\; \hbox{Minimize}\;\\[2pt]  \hspace{0.6cm}J(x,z,u,v):=g(x(0),x(1))+\frac{1}{2}\left(\|u(0)-\u(0)\|^2 + \|x(0)-\x(0)\|^2 +z(1)\right)\\[2pt] \hspace{0.6cm} \hbox{over}\;(x,z,u)\in \W^{1,1}([0,1];\R^n)\times \W^{1,1}([0,1];\R)\times \W^{1,1}([0,1];\R^m)\\[1pt] \hspace{0.55cm} \;\hbox{and measurable functions}\; v\colon[0,1]\f\R^m\;\hbox{such that}\\[3pt] \hspace{1.2cm} 
\begin{cases}\dot{x}(t)={f}_\Vphi(t, x(t),u(t))-\gk e^{\gk\psi(x(t))} \nabla\psi(x(t)),\;\, t\in[0,1]\;\textnormal{a.e.},\vspace{0.07cm}\\
\dot{u}(t)= v(t), \;\, t\in[0,1]\;\textnormal{a.e.},\vspace{0.07cm}\\
\dot{z}(t)=\|v(t)-\bar{v}(t)\|^2,\;\;t\in[0,1]\;\hbox{a.e.},
\\x(t)\in \bar{B}_{\delta}(\x(t))\;\hbox{and}\; u(t)\in U(t)\cap \bar{B}_{\delta}(\u(t)),\;\;\forall t\in[0,1],\\[1pt] 
(x(0), u(0),z(0))\in C_0(k)\times \R^m\times\{0\},\\[1pt]
(x(1), u(1),z(1))\in C_1(k)\times \R^m\times [-\delta,\delta].\end{cases}
 \end{array}$$ 

In the following lemma we apply to the above sequence of reformulated problems $(P_{\gk})$, where $k$ is as large as in Lemma \ref{lasthopefully},  the nonsmooth Pontryagin maximum principle for standard optimal control problems with {\it implicit} state constraints (see e.g., \cite[Theorem 9.3.1]{vinter} and \cite[p.332]{vinter}). For this purpose, $(x,z,u)$ is  the state function in $(P_{\gk})$ and $v$ is the control. Thus,  $(x_{\gk},z_{\gk},u_{\gk})$ is the optimal state, where $(x_{\gk},u_{\gk})$ is obtained from Lemma \ref{lasthopefully} and $z_{\gk}(t):=\int_0^t \|\dot{u}_{\gk}(s)-\dot{\u}(s)\|^2\;ds$, and $v_{\gk}=\dot{u}_{\gk}$ is the optimal control.

\begin{lemma} \label{mpw12aprox} For  $k$ large enough, there exist $\l_{\gk}\geq 0$, $p_{\gk}\in \W^{1,1}([0,1];\R^n)$, $q_{\gk}\in \W^{1,1}([0,1];\R^m)$,  $\Omega_{\gk} \in NBV([0,1];\R^m)$, $\mu^{o}_{\gk}\in \C^{\oplus}([0,1];\R^{m})$, and a $\mu_{\gk}^{o}$-integrable function $\beta_{\gk}\colon [0,1]\f\R^m$  such that  $\Omega_{\gk}(t)= \int_{[0,t]} \beta_{\gk}(s)\mu^{o}_{\gk}(ds)$, for all $t\in(0,1]$, and\sp$:$
\begin{enumerate}[$(i)$]
\item {\bf (The nontriviality condition)} For all $k\in\N$, we have $$\|p_{\gk}(1)\|+ \|q_{\gk}\|_\infty + \|\mu^{o}_{\gk}\|\TV + \l_{\gk}=1;$$
\item {\bf (The adjoint equation)} For \textnormal{a.e.} $t\in [0,1],$ 
\begin{eqnarray}\nonumber\left(
\begin{array}{c}
\dot{p}_{\gk}(t)\vspace{0.2cm}\\
\dot{q}_{\gk}(t)\\
\end{array}
\right)\in &-&\left(\partial^{\sp (x,u)} {f}_{\Vphi}(t,x_{\gk}(t),u_{\gk}(t))\right)\tran p_{\gk}(t)\\[3pt] &+& \label{adjapp} \left(
\begin{array}{c}
\gk e^{\gk \psi(x_{\gk}(t))}\partial^2\psi(x_{\gk}(t))p_{\gk}(t)\vspace{0.2cm}\\
0\\
\end{array}\right)\\[3pt] &+& \left(
\begin{array}{c}
\gk^2 e^{\gk \psi(x_{\gk}(t))}\nabla\psi(x_{\gk}(t))\<\nabla\psi(x_{\gk}(t)),p_{\gk}(t)\>\vspace{0.2cm}\\
0\\
\end{array}
\right);\nonumber\end{eqnarray}
\item {\bf (The transversality equation)} 
$$(p_{\gk}(0),-p_{\gk}(1))\in $$
\begin{equation*}\hspace{-0.4cm}\l_{\gk}\partial^L g(x_{\gk}(0),x_{\gk}(1))\bbp+\bbp\big[\big(\l_{\gk}(x_{\gk}(0)\bbp-\bbp\x(0)) \bbp+\bbp N_{C_0(k)}^L(x_{\gk}(0))\big) \bp \times \bp N_{C_1(k)}^L(x_{\gk}(1))\big],\vspace{0.2cm}\end{equation*} 
\hspace{-0.6cm} and \begin{equation*} q_{\gk}(0)=\l_{\gk}(u_{\gk}(0)-\u(0)),\;\;\;-q_{\gk}(1)=\Omega_{\gk}(1); \end{equation*}
\item {\bf (The maximization condition)} For \textnormal{a.e.} $t\in [0,1],$ \begin{equation*} \max_{v\in\R^m} \left\{\<q_{\gk}(t) +\Omega_{\gk}(t),v\>-\frac{\l_{\gk}}{2}\|v-\dot{\bar{u}}(t)\|^2\right\}  \;\hbox{is attained at}\;  \dot{u}_{\gk}(t);\end{equation*}   
\item {\bf (The measure properties)}
$$\supp\{\mu_{\gk}^{o}\}\subset  \left\{t\in [0,1] : (t,u_{\gk}(t))\in \bdry\Gr\left[U(t)\cap \bar{B}_{\delta}(\u(t))\right]\right\}, \;\; \mbox{and} $$ 
$$\beta_{\gk}(t)\in \partial_u^{\sp>} d(u_{\gk}(t),U(t)\cap \bar{B}_{\delta}(\u(t))) \;\;\;\mu_{\gk}^{o}\;\textnormal{a.e.,}$$
with  $\partial_u^{\sp>} d(u_{\gk}(t),U(t)\cap \bar{B}_{\delta}(\u(t))) \subset \left[\left(\conv\bar{N}^L_{U(t)\cap \bar{B}_{\delta}(\u(t))}(u_{\gk}(t))\right)\cap \left(\bar{B}\setminus\{0\}\right)\right]$.\end{enumerate}
\end{lemma}
\begin{proof} We intend to apply to  the optimal solution, $\left((x_{\gk},u_{\gk},z_{\gk}), v_{\gk}\right)$, of the reformulated $(P_{\gk})$, the {\it multiple state constraints} maximum principle \cite[p.331]{vinter} in which $$(h_1(t,x,u), h_2(t,x,u)):=  \left(d(x,\bar{B}_{\delta}(\x(t))), d(u,U(t)\cap \bar{B}_{\delta}(\u(t)))\right).$$

First, we show that the constraint qualification (CQ) that holds for $U(\cdot)$ at $\u$, also holds true at $u_{\gk}$, for $k$ large enough. Indeed, if this is false, then, by \cite[Proposition 2.3]{lowen91}, there exist an increasing sequence $(k_n)_n$ in $\N$ and a sequence $t_n\in [0,1]$ such that $t_n\f t_o\in [0,1]$ and \begin{equation} \label{CQ1} 0\in\partial_u^{\sp >}d(u_{\gkn}(t_n),U(t_n)),\;\;\forall n\in\N.\end{equation} The continuity of $\u$ and the uniform convergence of $u_{\gkn}$ to $\u$ yield that the sequence $(u_{\gkn}(t_n))_n$ converges to $\u(t_o)$. Hence, using that the multifunction $(t,x)\mapsto \partial_u^{\sp >}d(x,U(t))$ has closed values and a closed graph, we conclude from \eqref{CQ1} that $0\in \partial_u^{\sp >}d(\u(t_o),U(t_o))$. This contradicts that the constraint qualification is satisfied by $U(\cdot)$ at $\u$. Thus, for $k$ sufficiently large, $U(\cdot)$ satisfies the constraint qualification (CQ)  at $u_{\gk}$.

One can easily prove the lower semicontinuity of the multifunctions $t\mapsto \bar{B}_{\delta}(\x(t))$ and $t\mapsto \big[U(t)\cap \bar{B}_{\delta}(\u(t))\big]$, and hence, the  functions $h_1$ and $h_2$ that are Lipschitz  in $(x,u)$ are  lower semicontinuous in $(t,x,u)$.  A simple argument by contradiction that uses the uniform convergence of $u_{\gk}$ to $\u$, the local property of the limiting normal cone, and the constraint qualification (CQ) being satisfied by $U(\cdot)$ at $u_{\gk}$, yields that, for $k$ sufficiently large, the multifunction $U(\cdot)\cap \bar{B}_{\delta}(\u(\cdot))$ satisfies the constraint qualification (CQ) at $u_{\gk}$. Since $t\mapsto \bar{B}_{\delta}(\x(t))$ is lower semicontinuous  and its values are closed, convex, and have nonempty interior, and $x_{\gk}$ converges uniformly to $\x$, then we deduce that, for $k$ large enough, $\bar{B}_{\delta}(\x(\cdot))$ satisfies the constraint qualification (CQ) at $x_{\gk}$. As for all $t\in[0,1],$ $u_{\gk}(t)\in U(t),$  and by Theorem \ref{lasthopefully},  $x_{\gk}(t)\in\inte C$,  then (H1) yields that, for $t\in [0,1]$ a.e., we have ${f}(t,\cdot,\cdot)$ is $M_{\ell}$-Lipschitz on the {\it neighborhood}  of $(x_{\gk}(t),u_{\gk}(t))$. On the other hand, by Theorem \ref{lasthopefully}, we have, for  $k$ sufficiently large,  
\begin{eqnarray*}\nonumber (x_{\gk}(0), x_{\gk}(1))&\in& \big[\left(C_0\cap \bar{B}_{\delta_{\bbbp o}}(\x(0))\right)+\tilde{\rho}{B}\big]\cap(\inte C) \times \big[\left(C_1\cap \bar{B}_{\delta_{\bbbp o}}(\x(1))\right)+\tilde{\rho}{B}\big]\cap(\inte C) \\ &\subset&\inte (\tilde{C}_0(\delta) \times \tilde{C}_1(\delta)). \label{inta5}\end{eqnarray*}
Therefore, the data of $(P_{\gk})$ satisfy all the hypotheses of the maximum principle stated in \cite[p.331]{vinter}, which is deduced from  \cite[Theorem 9.3.1]{vinter}, by taking the  scalar state constraint function  therein to be   $h(t,x,u)=\max\{h_1(t,x,u),h_2(t,x,u)\}$.
When applying  that maximum principle to  $(P_{\gk})$ at the optimal solution $\left((x_{\gk},z_{\gk},u_{\gk}),v_{\gk}\right)$, we notice that:
\begin{itemize}
\item As in Step 2 of the proof of \cite[Theorem 5.1]{verachadi},  the adjoint variable $p_{\gk}$ corresponding to $x_{\gk}$ satisfies the adjoint equation  that is linear in $p_{\gk}$ and: there exists   $M_1>0$ such that \begin{equation*}\label{ntc1} \|p_{\gk}(t)\|\leq M_1\|p_{\gk}(1)\|\;\;\hbox{for all}\;t\in[0,1]. \end{equation*}
This gives that $p_{\gk}=0$ if and only if $p_{\gk}(1)=0$. Therefore, in the nontriviality condition, $\|p_{\gk}\|_\infty$ can be replaced by $\|p_{\gk}(1)\|.$
\item From Remark (a) on \cite[page 330]{vinter}, the set $I(\x)$  in the statement of \cite[Theorem 9.3.1]{vinter}, and hence in that of \cite[p.331]{vinter} for $h_1$ and $h_2$,   can be replaced,  by 
\begin{equation}\label{supp}\{t\in [0,1] : \partial_{x}^{\sp>} h(t,\x(t))\not=\emptyset\}.
\end{equation}
Moreover, if $h(t,x):=d(x,F(t))$, where $F\colon [0,1]\rightrightarrows\R^m$ be a lower semicontinuous multifunction with closed and nonempty values, then by \cite[Proposition 2.3(a) $\&$ Equation (2.15)]{lowen91} we have that 
\begin{equation}\label{del>} \{t\in [0,1] : \partial_{x}^{\sp>} h(t,\x(t))\not=\emptyset\}=\{t\in [0,1] : (t,\x(t))\in \bdry\Gr F(t)\}.\end{equation}
\item The measure corresponding to the state constraint ``$\sp x(t)\in \bar B_{\delta}(\x(t))$ for all $t\in [0,1]$" (or equivalently ``$h_1(t,x(t),u(t))\leq 0$\sp") is {\it null}. This is  due to the fact,  from Theorem \ref{lasthopefully},  that  for $k$ sufficiently large, $x_{\gk}(t)\in B_{\delta}(\x(t))$ for all $t\in [0,1]$, which gives, for $k$ sufficiently large, that this measure is supported in \begin{eqnarray*} && \{t\in [0,1] : \partial_{(x,u)}^{\sp>} h_1(t,x_{\gk}(t),u_{\gk}(t))\not=\emptyset\}\\[3pt] &\overset{\eqref{del>}}{=}& \{t\in [0,1] : (t,x_{\gk}(t))\in \bdry\Gr B_{\delta}(\x(t))\}\\[3pt]&=&\left\{t\in [0,1] : (t,x_{\gk}(t))\in \bigcup_{t\in [0,1]} \{t\}\times S_{\delta}(\x(t))\right\}= \emptyset, \end{eqnarray*}
where $S_{\delta}(\x(t)):=\{x\in \R^n: \|x-\x(t)\|=\delta\}.$ 

\item The adjoint vector $e_{\gk}$ corresponding to the optimal state $z_{\gk}$ is the constant $-\frac{\l_{\gk}}{2}$ (where $\l_{\gk}$ is the cost multiplier). Indeed, since $v_{\gk}$ converges strongly in $L^2$ to $\bar{v}$, we have, for $k$ sufficiently large, that $z_{\gk}(1)\in [0,\delta)\subset \inte ([-\delta,\delta])$. Add to this that $\dot{e}_{\gk}(t)=0$ for $t\in[0,1]$ a.e.,  and by the transversality condition we get  $$\textstyle e_{\gk}(t)= e_{\gk}(1)\in -\left\{\frac{\l_{\gk}}{2}\right\}- N^L_{[-\delta,\delta]}(z_{\gk}(1))=-\left\{\frac{\l_{\gk}}{2}\right\},\;\;\forall t\in [0,1].$$ Hence, for $k$ sufficiently large, $e_{\gk}(t)=-\frac{\l_{\gk}}{2}$ for all $t\in [0,1]$.
\item  The $BV$-function associated  to the state constraint ``$\sp u(t)\in U(t)\cap \bar B_{\delta}(\u(t))$" (or equivalently ``$h_2(t,x(t),u(t))\leq 0$\sp") in the multiple state maximum principle takes the form $\int_{[0,t)} \beta_{\gk}(t) \mu^{\bbbp o}_{\gk}(dt)$, for $t\in[0,1)$ and $\int_{[0,1]} \beta_{\gk}(t) \mu^{\bbbp o}_{\gk}(dt)$, for $t=1$, where,   by  \eqref{supp} and \eqref{del>},
$$\supp\{\mu_{\gk}^{\bbbp o}\}\subset \left\{t\in [0,1] : (t,u_{\gk}(t))\in \bdry\Gr\left[U(t)\cap \bar{B}_{\delta}(\u(t))\right]\right\},$$
$$\beta_{\gk}(t)\in \partial_u^{>}d(u_{\gk}(t), U(t)\cap \bar{B}_{\delta}(\u(t))) \;\;\;\mu_{\gk}^{\bbbp o}\;\textnormal{a.e.,}$$
 and, by  the (CQ) property  and \cite[Formula (9.17)]{vinter},
 $$\partial_u^{>}d(u_{\gk}(t), U(t)\cap \bar{B}_{\delta}(\u(t))) \subset\left[\conv\bar{N}^L_{U(t)\cap \bar{B}_{\delta}(\u(t))}(u_{\gk}(t))\cap \left(\bar{B}\setminus\{0\}\right)\right].$$
 However, with a simple normalization procedure (see e.g., the relevant part in the proof of \cite[Theorem 3.4]{hoehener}) we can easily obtain a function  $\Omega_{\gk}\in NBV([0,1];\R^m)$  satisfying, together with $\beta_{\gk}$ and $\mu^{\bbbp o}_{\gk}$, $\Omega_{\gk}(t)= \int_{[0,t] }\beta_{\gk}(t) \mu^{o}_{\gk}(dt)$, for $t\in(0,1]$, $\Omega_{\gk}(0) =0$,  and the statement of the multiple state maximum principle remains valid with this function $\Omega_{\gk}$.

\end{itemize}
Therefore, we obtain the existence of $\l_{\gk}\geq 0$, $p_{\gk}\in \W^{1,1}([0,1];\R^n)$, $q_{\gk}\in \W^{1,1}([0,1];\R^m)$,  $\Omega_{\gk} \in NBV([0,1];\R^m)$, $\mu^{\bbbp o}_{\gk}\in\C^{\oplus}([0,1];\R^{m})$, and a Borel measurable function $\beta_{\gk}\colon [0,1]\f\R^m$  such that    $\Omega_{\gk}(t)= \int_{[0,t]} \beta_{\gk}(s)\mu^{\bbbp o}_{\gk}(ds)$ for  all $t\in(0,1]$, $\Omega_{\gk}(0)=0$, and conditions $(i)$-$(v)$ of this lemma hold. Note that in the adjoint equation \eqref{adjapp}, the values of $f(t,\cdot,\cdot)$ outside the set $\left[C \cap \bar{B}_{\delta}(\x(t))\right]\times \big[(U(t)+\tilde{\rho}\bar{B})\cap  \bar{B}_{\delta}(\u(t))\big]$ are not involved in the calculation of the subdifferential $\partial^{\sp (x,u)} {f}_{\Vphi}(t,x_{\gk}(t),u_{\gk}(t))$ since $(x_{\gk}(t),u_{\gk}(t))$ belongs to the {\it interior} of that set, for $k$ sufficiently large and for all $t\in [0,1]$.  \end{proof}

\subsection*{Step 3: Finalizing the proof} Since, by Lemma \ref{mpw12aprox}, conditions $(i)$-$(v)$ of \cite[Proposition 4.7]{ZNpaper} are valid, then to terminate the proof of  Theorem \ref{mpw12bis}, it is sufficient to follow the proof of \cite[Theorem 4.8]{ZNpaper}.

\section{Example} \label{example} In this section, we present an example in which we illustrate how Theorem \ref{mpw12bis}
can be used to find an optimal solution when the perturbation function is {\it nonautonomous}. We consider the problem $(P)$ in which (see Figure \ref{Fig1}):
\begin{itemize}
\item The {\it nonautonomous} perturbation mapping $f\colon[0,\frac{\pi}{2}]\times\R^2\times\R\f\R^2$  is defined by $$f(t,(x_1,x_2),u)=(t-x_1-x_2-u,-t+x_1-x_2+u).$$
\item The function $\psi\colon \R^2\f\R$ is defined by $\psi(x_1,x_2):=(x_1^2+x_2^2-1)(x_1^2+x_2^2-4)$. Hence, the set $C$ is  the  nonconvex and compact  set $$C:=\{(x_1,x_2) : (x_1^2+x_2^2-1)(x_1^2+x_2^2-4)\leq 0\}.$$
\item The objective function $g\colon\R^4\f\R\cup\{\infty\}$  is defined  by $$g(x_1,x_2,x_3,x_4):=\begin{cases} \frac{1}{2}(x_3^2+x_4^2-1)& \;\;(x_3,x_4)\in C, \vspace{0.1cm}\\ \infty&\;\;\hbox{Otherwise}. \end{cases}$$
\item The function $\varphi$ is the indicator function of $C$. 
\item The control multifunction is $U(t):=[t,\pi]$ for all $t\in [0,\frac{\pi}{2}]$.
\item The two sets  $C_0$ and $C_1$ are defined by $C_0:=\{(1,0)\}$ and $C_1:=\{(0,x_2) : x_2\geq 0\}.$
\end{itemize}
One can easily verify that the hypotheses (H2)-(H4.4) are satisfied.  Add to this that $f(t,\cdot,\cdot)$ is globally Lipschitz on $\R^2\times\R$, $g$ is globally Lipschitz on $\R^2\times C$, and $U(t)$ is convex with nonempty interior, we deduce that all the hypotheses of Theorem \ref{mpw12bis} are satisfied. Since the function $g$ vanishes on the unit circle and is strictly positive elsewhere in $C$, we may seek  for $(P)$ an optimal solution $(\x,\u)$ such that, if possible,  $\x:=(\x_1,\x_2)$ belongs to the unit circle,  and hence  we have 
\begin{equation} \label{ex1}\begin{cases} \x_1^2(t)+\x_2^2(t)=1,\,\forall t\in [0,\tfrac{\pi}{2}];\,\hbox{and}\;\,\x_1(t)\dot{\x}_1(t)+\x_2(t) \dot{\x}_2(t)=0,\; \forall t\in[0,\frac{\pi}{2}] \;\mbox{a.e.,}\\[1pt] \x(0)\tran=(1,0)\,\;\hbox{and}\;\,\x(\tfrac{\pi}{2})\tran=(0,1).\end{cases} \end{equation} 
\begin{figure}[t]
\centering
\includegraphics[scale=0.4]{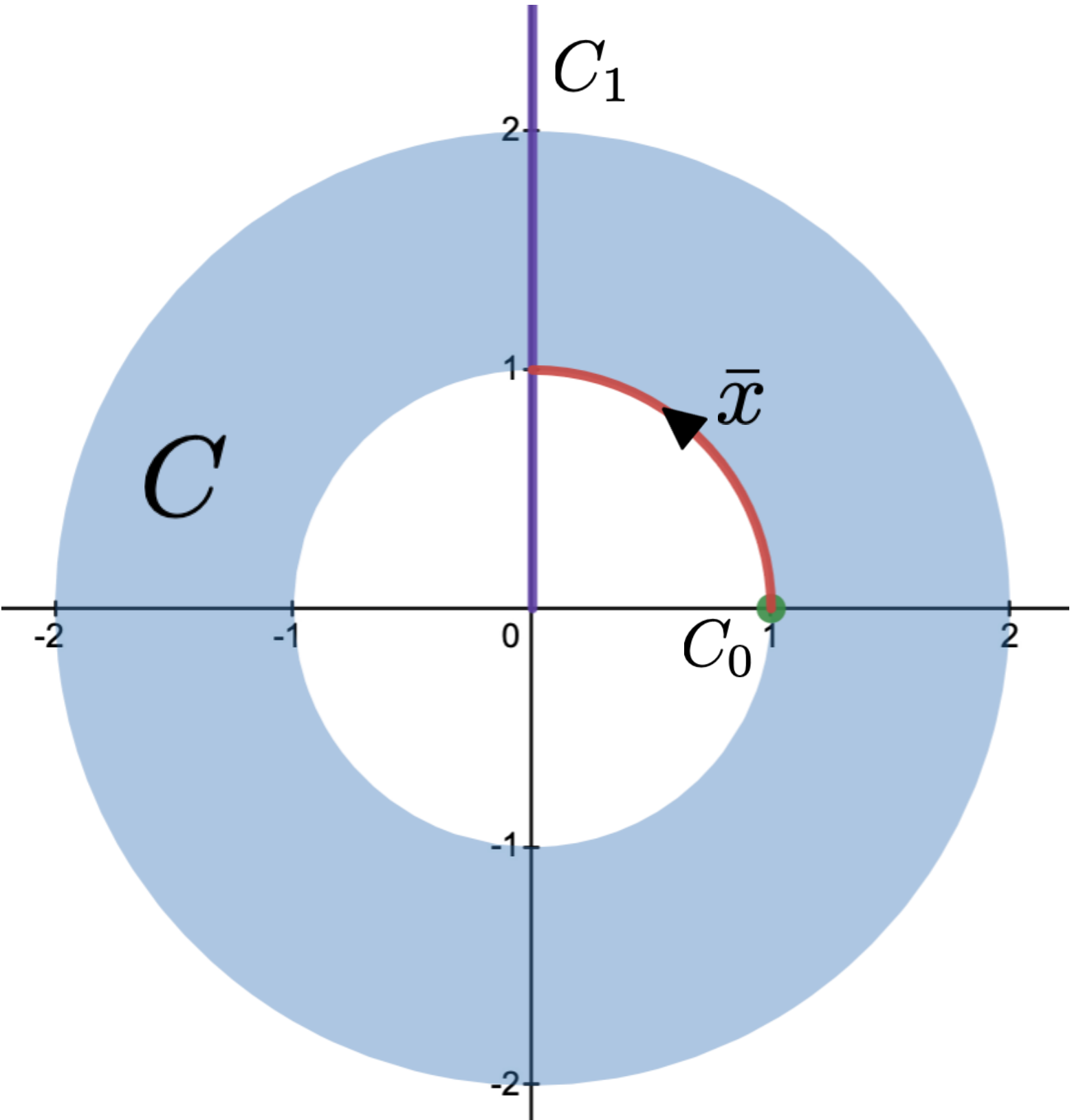}
\caption{\label{Fig1} Example \ref{example}}
\end{figure}
Applying Theorem \ref{mpw12bis} to this optimal solution $(\x,\u)$ and using Remark \ref{specialcases}$(c)$,   we get the existence of an adjoint vector $\p:=(\p_1,\p_2)\in BV([0,\frac{\pi}{2}];\R^2)$, a finite signed Radon measure $\bar{\nu}$ on $\left[0,\frac{\pi}{2}\right]$, $\xxi\in L^{\infty}([0,\tfrac{\pi}{2}];\R^{+})$, and $\l\geq 0$ such that, when incorporating  equations \eqref{ex1}  into $(i)$- $(vi)$, these latter simplify to the following:
\begin{enumerate}[(a)]
\item $\|\p(\frac{\pi}{2})\|+\l=1.$
\item The admissibility equation holds, that is,  for  $t\in[0,\frac{\pi}{2}]$ a.e.,  \begin{equation*}\begin{cases}\dot{\bar{x}}_1(t)=t -\x_1(t)-\x_2(t)-\u+6\x_1(t)\bar{\xi}(t),\\ 
\dot{\bar{x}}_2(t)=-t+ \x_1(t)-\x_2(t)+\u+6\x_2(t)\bar{\xi}(t). \end{cases} \end{equation*} 
\item The adjoint equation is satisfied, that is, for  $t\in [0,\frac{\pi}{2}]$, 
\begin{eqnarray*} d\p(t)&=&\begin{pmatrix} 
1\;\; & -1 \\
1\;\; & 1 
\end{pmatrix}\p(t)\sp dt\sp+\sp \bar{\xi}(t)\begin{pmatrix} 
8\x_1^2(t)-6\;\; & 8\x_1(t)\x_2(t)  \\
8\x_1(t)\x_2(t) \;\; & 8\x_2^2(t)-6 \end{pmatrix}\p(t)\sp dt\sp\\&-&6\sp \begin{pmatrix} 
\x_1(t) \\
\x_2(t)
\end{pmatrix} d\bar{\nu}.\end{eqnarray*}
\item The complementary slackness condition is valid, that is,   $$\bar{\xi}(t)(\p_1(t)\x_1(t) + \p_2(t)\x_2(t))=0,  \:\: t\in \left[0,\tfrac{\pi}{2}\right]\; \mbox{a.e.}$$
\item The transversality condition holds: $-\p(\frac{\pi}{2})\in \l\{(0,1)\}+\{(\a,0)\in\R^2 : \a\in\R\}$
\item $\max\{u(\p_2(t)-\p_1(t)): u\in [t,\pi]\}$ is attained at $\u(t)$ for $t\in[0,\frac{\pi}{2}]$ a.e.
\end{enumerate}
From  \eqref{ex1} combined with (b), we deduce that 
\begin{equation} \label{ex2} \bar{\xi}(t)=\frac{1+(\u(t)-t)(\x_1(t)-\x_2(t))}{6},\;\textstyle \forall t\in[0,\frac{\pi}{2}].\end{equation}
On the other hand,  the use of (d) and \eqref{ex1} in (c), yields that, for  $t\in [0,\frac{\pi}{2}]$,   \begin{equation}\label{ex3} \begin{cases}d{\bar{p}}_1= (\p_1(t)-\p_2(t)-6\bar{\xi}(t)\p_1(t))\sp dt-6\x_1(t)\sp d\bar{\nu},\\d{\bar{p}}_2= (\p_1(t)+\p_2(t)-6\bar{\xi}(t)\p_2(t))\sp dt-6\x_2(t)\sp d\bar{\nu}.  \end{cases} \end{equation} 
Now in order to exploit (f), we temporarily assume that \begin{equation}\label{ex4} \p_2(t)<\p_1(t)\;\hbox{for}\; t\in[0,\tfrac{\pi}{2}]\;\hbox{a.e.,}\end{equation} 
hoping to be able to find $\p_1$ and $\p_2$ satisfying this condition. In this case, $\u(t)=t$ for all $t\in [0,\frac{\pi}{2}]$, which gives using \eqref{ex2} that $\bar{\xi}(t)=\frac{1}{6}$ for all $t\in \left[0,\frac{\pi}{2}\right].$ Using these  values of $\u$ and $\bar{\xi}$,  and \eqref{ex1},  we can solve for $(\x_1,\x_2)$ the two differential equations of (b)  to obtain that $$\textstyle \x(t)\tran=(\cos t,\sin t),\;\;\forall t\in[0,\frac{\pi}{2}].$$
Employing (a), (d), (e), and \eqref{ex3}, a simple calculation yields that $$\begin{cases} \l=\tfrac{3}{8}\;\,\hbox{and}\;\,\p(\tfrac{\pi}{2})=(\tfrac{1}{2},-\tfrac{3}{8}),\\[2pt]\p(t)\tran=\tfrac{1}{2}(\sin t ,-\cos t)\;\hbox{on}\;[0,\tfrac{\pi}{2})\,\;\hbox{and}\;\,d\bar{\nu}=\tfrac{1}{16}\sp\delta_{\big\{\bbp\tfrac{\pi}{2}\bbp\big\}}\;\hbox{on}\;[0,\tfrac{\pi}{2}],\end{cases}\vspace{0.01cm}$$
where $\delta_{\{a\}}$ denotes the unit measure concentrated on the point $a$. Note that for all $t\in [0,\tfrac{\pi}{2}]$, we have $\p_2(t)<\p_1(t) $, and hence, the temporary assumption \eqref{ex4} is satisfied. Therefore, the above analysis, realized via Theorem \ref{mpw12bis}, produces  an admissible pair $(\x,\u)$, where $$\textstyle \x(t)\tran=(\cos t,\sin t)\;\;\hbox{and}\;\;\u(t)=t,\;\;\forall t\in [0,\frac{\pi}{2}],$$
which is optimal for $(P)$.

\end{document}